\documentclass[12pt]{article}

\usepackage{amsmath}
\usepackage{amsthm}
\usepackage{amsfonts}
\usepackage{amssymb}
\usepackage{graphicx}
\usepackage{latexsym}
\usepackage{mathrsfs}
\usepackage{epsfig}

\advance \topmargin by -\headheight
\advance \topmargin by -\headsep
\evensidemargin \oddsidemargin
\newtheorem{theorem}{Theorem}[section]
\newtheorem{thm}{Theorem}[section]

\newtheorem{lem}[theorem]{Lemma}

\newtheorem{prop}[theorem]{Proposition}

\newtheorem{convention}[theorem]{Convention}

\theoremstyle{definition}

\newtheorem{defn}[theorem]{Definition}
\newtheorem{example}[theorem]{Example}

\newtheorem{cor}[theorem]{Corollary}

\theoremstyle{remark}
\newtheorem{remark}[theorem]{Remark}

\numberwithin{equation}{section}

\newcommand\NN{\mathbb N}
\newcommand\RR{\mathbb R}
\newcommand\ZZ{\mathbb Z}

\newcommand\cA{\mathcal{A}}

\def\cc{\curvearrowright}

\def\fh{{\mathfrak{v}}}

\def\sW{{\mathscr{W}}}
\def\sX{{\mathscr{X}}}
\def\sY{{\mathscr{Y}}}
\def\sZ{{\mathscr{Z}}}

\def\N{{\mathbb{N}}}

\begin{document}
%\title{The lattice point counting problem with
%uniform error estimates}
\title{The type and stable type of the boundary of a Gromov hyperbolic group}

\author{Lewis Bowen\footnote{supported in part by NSF grant DMS-0968762, NSF CAREER Award DMS-0954606 and BSF grant 2008274} }% ~and Amos Nevo\footnote{supported in part ISF grant  and BSF grant 2008274}}%\address{University of Hawaii}
%\email{lpbowen@math.hawaii.edu}
%\thanks{The first author was supported in part by NSF Grant}


%\address{Institute for Advanced Study, Princeton}

%\curraddr{Department of Mathematics, Technion}
%\email{anevo@tx.technion.ac.il, nevo@math.ias.edu}
%\thanks{The second author was supported by the ISF}

%    General info
%\subjclass{Primary ; Secondary}

%\date{February 2011}

%\dedicatory{}

%\keywords{}
\maketitle

\begin{abstract}
Consider an ergodic non-singular action $\Gamma \cc B$ of a countable group on a probability space. The type of this action codes the asymptotic range of the Radon-Nikodym derivative, also called the {\em ratio set}. If $\Gamma \cc X$ is a pmp (probability-measure-preserving) action, then the ratio set of the product action $\Gamma \cc B\times X$ is contained in the ratio set of $\Gamma \cc B$. So we define the {\em stable ratio set} of $\Gamma \cc B$ to be the intersection over all pmp actions $\Gamma \cc X$ of the ratio sets of $\Gamma \cc B\times X$. By analogy, there is a notion of {\em stable type} which codes the stable ratio set of $\Gamma \cc B$. This concept is crucially important for the identification of the limit in pointwise ergodic theorems established by the author and Amos Nevo. 

Here, we establish a general criteria for a nonsingular action of a countable group on a probability space to have stable type $III_\lambda$ for some $\lambda >0$. This is applied to show that the action of a non-elementary Gromov hyperbolic group on its boundary with respect to a quasi-conformal measure is not type $III_0$ and, if it is weakly mixing, then it is not stable type $III_0$. 
\end{abstract}

\tableofcontents

\section{Introduction}

\subsection{Ratio set}
Let $\Gamma$ be a countable group and $\Gamma \cc (B,\nu)$ a nonsingular action on a standard probability space. The {\em ratio set} of the action, denoted $RS(\Gamma \cc (B,\nu)) \subset [0,\infty]$, is defined as follows. A real number $r\ge 0$ is in $RS(\Gamma \cc (B,\nu)$ if and only if for every positive measure set $A \subset B$ and $\epsilon>0$ there is a subset $A' \subset A$ of positive measure and an element $g\in \Gamma \setminus\{e\}$ such that 
\begin{itemize}
\item $gA' \subset A$,
\item $| \frac{d\nu\circ g}{d\nu}(b)-r| < \epsilon$ for every $b \in A'$.
\end{itemize}
The extended real number $+\infty \in RS(\Gamma \cc (B,\nu))$ if and only if for every positive measure set $A \subset B$ and $n>0$ there is a subset $A' \subset A$ of positive measure and an element $g\in \Gamma \setminus\{e\}$ such that 
\begin{itemize}
\item $gA' \subset A$,
\item $ \frac{d\nu\circ g }{d\nu}(b) > n$ for every $b \in A'$.
\end{itemize}
The ratio set is also called the {\em asymptotic range} or {\em asymptotic ratio set}. By Proposition 8.5 of \cite{FM77}, if the action $\Gamma \cc (B,\nu)$ is ergodic then $RS(\Gamma \cc (B,\nu))$ is a closed subset of $[0,\infty]$. Moreover, $RS(\Gamma \cc (B,\nu)) \setminus \{0,\infty\}$ is a multiplicative subgroup of $\RR_{>0}$. Since 
$$\frac{d\nu\circ g^{-1}}{d\nu}(gb) = \left( \frac{d\nu\circ g}{d\nu}(b) \right)^{-1},$$
the number $0$ is in the ratio set if and only if $\infty$ is in the ratio set. So if $\Gamma \cc (B,\nu)$ is ergodic and non-atomic then the possibilities for the ratio set and the corresponding type classification are:
\begin{displaymath}
\begin{array}{c|c}
\textrm{ ratio set } & \textrm{ type }\\
\hline
\{1\} & II\\
\{0,1,\infty\} & III_0\\
\{0,\lambda^n, \infty: ~n\in \ZZ\} & III_\lambda\\
 \lbrack 0,\infty \rbrack & III_1
 \end{array}\end{displaymath}
 where $\lambda \in (0,1)$. For a very readable review, see \cite{KW91}. Section 8 of \cite{FM77} discusses these concepts in the more general setting of cocycles taking values in an arbitrary locally compact group. 

\subsection{Stable ratio set}
Observe that if $\Gamma \cc (X,\mu)$ is a probability-measure-preserving (pmp) action then the ratio set of the product action satisfies $RS(\Gamma \cc (B\times X, \nu \times \mu)) \subset RS(\Gamma \cc (B,\nu))$. Therefore, it makes sense to define the {\em stable} ratio set of $\Gamma \cc (B,\nu)$ by 
$$SRS(\Gamma \cc (B,\nu)) = \bigcap RS(\Gamma \cc ( B\times X,\nu \times \mu))$$
where the intersection is over all pmp actions $G \cc (X,\mu)$. 

We say that $\Gamma \cc (B,\nu)$ is {\em weakly mixing} if for any ergodic pmp action $\Gamma \cc (X,\mu)$, the product action $\Gamma \cc (B\times X,\nu\times\mu)$ is ergodic. In this case, $SRS(\Gamma \cc (B,\nu))$ is a closed subset of $[0,\infty]$ and $SRS(\Gamma \cc (B,\nu)) \setminus \{0,\infty\}$ is a multiplicative subgroup of $\RR_{>0}$. So the possibilities for the stable ratio set are the same as those for the ratio set and we define the stable type according. For example, if $SRS(\Gamma \cc (B,\nu)) =\{0,\lambda^n, +\infty: ~n\in \ZZ\}$ for some $\lambda \in (0,1)$ then we say the stable type of the action is $III_\lambda$.

\subsection{The main result}

The main result of this paper is:
\begin{thm}\label{thm:main}
Let $(\Gamma,d)$ be a non-elementary uniformly quasi-geodesic Gromov hyperbolic group. Then the stable ratio set of the action of $\Gamma$ on its Gromov boundary with respect to any quasi-conformal measure contains an element of the form $e^T$ for some $0<T<\infty$. Therefore, it does not have type $III_0$. If it is weakly mixing then it is not stable type $III_0$.
\end{thm}
See \S \ref{sec:Gromov} for the definitions of uniformly quasi-geodesic Gromov hyperbolic group, Gromov boundary and quasi-conformal measure. It appears to be unknown whether the action of $\Gamma$ on its boundary is always weakly mixing. %We will show in a future publication that if the horofunction boundary of $(\Gamma,d)$ coincides with its Gromov boundary then the action of $\Gamma$ on its boundary is weakly mixing.

By comparison, in \cite{INO08} it is proven that the Poisson boundary of a random walk on a Gromov hyperbolic group induced by a nondegenerate measure on $\Gamma$ of finite support is never of type $III_0$. In \cite{Su78, Su82}, Sullivan proved that the recurrent part of an action of a discrete conformal group on the sphere $\mathbb{S}^d$ relative to the Lebesgue measure is type $III_1$. Spatzier \cite{Sp87} showed that if $\Gamma$ is the fundamental group of a compact connected negatively curved manifold then the action of $\Gamma$ on the sphere at infinity of the universal cover is also of $III_1$. The types of harmonic measures on free groups were computed by Ramagge and Robertson \cite{RR97} and Okayasu \cite{Ok03}. 

In \cite{BN2} it shown that if $\Gamma$ is an irreducible lattice in a connected semisimple Lie group $G$ which has trivial center and no compact factors, then the action of $\Gamma$ on $(G/P,\nu)$ is stable type $III_1$ (where $P<G$ is a minimal parabolic subgroup and $\nu$ is a probability measure in the unique $G$-invariant measure class). I do not know of any other results on stable type.

\subsection{The Maharam extension and ergodic theorems}

The main reason for the interest in type and stable type is because of its connection with the  {\em Maharam extension} (of $\Gamma \cc (B,\nu)$) which is  the action of $\Gamma$ on $B\times \RR$ given by
$$g(b,t) = \left(gb, t - \log\left(\frac{d\nu\circ g}{d\nu}(b)\right)\right).$$
This action preserves the measure $\nu\times \theta$ where $\theta$ is the measure on $\RR$ given by $d\theta(t)=e^t dt$. Even if $\Gamma \cc (B,\nu)$ is ergodic, the action $\Gamma \cc (B\times \RR,\nu\times\theta)$ might not be ergodic. The type of an action quantifies how ergodic its Maharam extension is.

To be precise, the real line also acts on $B\times \RR$ by $s(b,t)=(b,t+s)$. This action does not preserve $\nu\times \theta$ but it does preserve its measure class. Moreover this action commutes with the $\Gamma$-action. Therefore, it descends to an action on the space $(Z,\zeta)$ of ergodic components of $\Gamma \cc (B\times \RR,\nu\times\theta)$, also called the {\em Mackey range} of the Radon-Nikodym cocycle. The action $\RR \cc (Z,\zeta)$ is trivial if and only if $\Gamma \cc (B,\nu)$ has type $III_1$. $\RR \cc (Z,\zeta)$ is equivalent to the action of $\RR$ on the circle $\RR/(-\log \lambda) \ZZ$  for some $\lambda \in (0,1)$ if and only if the type is $III_\lambda$. In all other cases, the type is $III_0$ (assuming $\Gamma \cc (B,\nu)$ is ergodic). 

Stable type was introduced in recent work \cite{BN2} (see also \cite{BN2b}) which proves general pointwise ergodic theorems for pmp actions of a countable group $\Gamma$. Roughly speaking, the idea is to consider an amenable action $\Gamma \cc (B,\nu)$. From a F\o lner sequence for a reduction of the Maharam extension of this action, one can construct a pointwise convergent sequence of probability measures on $\Gamma$. That is, a sequence $\{\pi_i\}_{i=1}^\infty$ of probability measures on $\Gamma$ is constructed such that if $\Gamma \cc (X,\mu)$ is any pmp action and $f \in L^p(X,\mu)$ (for $p>1$) then the averages $\pi_i f : = \sum_{g\in \Gamma} \pi_i(g) f \circ g^{-1}$ converge pointwise almost everywhere. However, in order to show that $\pi_i f$ converges to a $\Gamma$-invariant function, one needs to assume that $\Gamma \cc (B,\nu)$ has stable type $III_\lambda$ for some $\lambda \in (0,1]$. 

Amos Nevo and I are currently preparing an article on pointwise ergodic theorems for Gromov hyperbolic groups using the main results of this paper and \cite{BN2,BN2b}. 

\subsection{Organization and discussion of proofs}

The classification of ratio sets given above was obtained under the hypothesis that $\Gamma \cc (B,\nu)$ is ergodic. We remove this assumption in \S \ref{sec:nonergodic}. This easily proven result is helpful for establishing a general criterion (Theorem \ref{thm:AvEP}) for proving that a given nonsingular action $\Gamma \cc (B,\nu)$ contains a number $x$ in its stable range with $x \notin \{0,1,\infty\}$. The criterion is that there exists a sequence of integral kernels $\Upsilon_n:\Gamma \times B \times B \to [0,1]$ satisfying various conditions. Roughly speaking the idea is that if for every $b \in B$ there exists $b' \in B$ that is close to $b$ and $g\in \Gamma$ such that the logarithmic derivatives $\log\frac{d\nu \circ g}{d\nu}(b), \log\frac{d\nu \circ g}{d\nu}(b')$ are uniformly bounded but also separated from each other then there should exist a number in the ratio set corresponding to their difference. To put this rough idea into effect, instead of a sequence of maps from $B$ to $\Gamma \times B$, it is more convenient to consider a sequence of maps from $B$ to the space of probability measures on $\Gamma \times B$. %The formulation in \S \ref{sec:AvEP} is equivalent to this.

A curious feature of this criterion is that it works entirely with $(B,\nu)$. One does not have to consider pmp actions $\Gamma \cc (X,\mu)$ explicitly. This is reflected in the proof - which first establishes the existence of an interesting number $x$ in the ratio set (by interesting, we mean a number not contained in $\{0,1,\infty\}$) and then shows that $x$ must also be contained in the stable ratio set. The latter is accomplished by showing that any sequence of integral kernels satisfying the appropriate conditions for $\Gamma \cc (B,\nu)$ can be extended to a sequence of integral kernels for the product action $\Gamma \cc (B\times X,\nu\times \mu)$. 

After the general criterion is established, it remains to verify its conditions for Gromov hyperbolic groups. In \S \ref{sec:Gromov}, we review hyperbolic groups, quasi-conformal measures and set notation. We give an explicit family of integral kernels in \S \ref{sec:automatic} and proceed to verify that they satisfy the conditions of Theorem \ref{thm:AvEP}. This part of the proof uses estimates on the cardinalities of the intersections of spheres and horospheres and a variety of quasi-identities from hyperbolic geometry. Theorem \ref{thm:main} follows from Theorem \ref{thm:AvEP}, Corollary \ref{cor:AvEP} and Lemma \ref{lem:admissible}.

\subsection{An example}
We close this introduction with an example showing that the action of the free group on its Gromov boundary with the usual conformal measure enjoys the curious property that its type is not equal to its stable type. 

\begin{example}[Free groups]
Let $F_r = \langle s_1,\ldots, s_r\rangle$ be the free group of rank $r \ge 2$. Let $S=\{s_1,\ldots, s_r, s_1^{-1},\ldots s_r^{-1}\}$. The boundary of $F_r$, denoted $\partial F_r$ is the set of all half-infinite sequences $\xi=(\xi_1,\xi_2,\ldots) \in S^\N$ with $\xi_{i+1} \ne \xi_i^{-1}$ for all $i\ge 1$. There is an obvious Markov measure $\nu$ on $\partial F_r$ determined by the following. For every $(\xi_1,\ldots, \xi_n) \in S^n$ satisfying $\xi_{i+1}\ne \xi_i^{-1}$ for all $1\le i \le n-1$, 
$$\nu( \{ \xi' \in \partial F_r:~ \xi'_i = \xi_i ~\forall 1\le i \le n\}) = (2r)^{-1}(2r-1)^{-n+1}.$$
The {\em reduced form} of an element $g\in F_r$ is the expression  $g=t_1\cdots t_n$ with $t_i \in S$ and $t_{i+1}\ne t_i^{-1}$ for all $i$. It is unique. There is a natural action of $F_r$ on $\partial F_r$ by 
$$(t_1\cdots t_n)\xi := (t_1,\ldots,t_{n-k},\xi_{k+1},\xi_{k+2}, \ldots)$$ where $t_1,\ldots, t_n \in S$,  $t_1\cdots t_n$ is in reduced form and $k$ is the largest number $\le n$ such that $\xi_i^{-1} = t_{n+1-i}$ for all $i\le k$.  Observe that if $g=t_1 \cdots t_n$ then the Radon-Nikodym derivative satisfies
$$\frac{d\nu \circ g}{d\nu}(\xi) = (2r-1)^{2k-n}.$$
From this it is not difficult to show that the action of $F_r$ on its boundary is type $III_\lambda$ with $\lambda = (2r-1)^{-1}$. However, the stable type of the action is different. Indeed, consider the action of $F_r$ on $\ZZ/2\ZZ$ in which every generator in $S$ acts nontrivially. This action preserves the uniform probability measure on $\ZZ/2\ZZ$. If $A \subset B \times \{ 0\} \subset B \times \ZZ/2\ZZ$ then for any $a\in A$, any element $g \in F_r$ with $ga \in A$ has even word length. Therefore, the Radon-Nikodym derivative $\frac{d\nu \circ g}{d\nu}(a)$ is an integral power of $(2r-1)^{-2}$. This shows that the stable ratio set does not contain $(2r-1)^{-1}$. It is a nontrivial fact (proven implicitly in \cite{BN1}) that the action $\Gamma \cc (B,\nu)$ has stable type $III_\tau$ with $\tau = (2r-1)^{-2} = \lambda^2$. 
\end{example}

{\bf Acknowledgements.}
I am grateful to Chris Connell for showing me how to prove Lemma \ref{lem:vol2}. This paper owes a debt of gratitude to Amos Nevo for many conversations about many previous versions of these results and their proofs, which have considerably simplified over several years.

\section{The ratio set of a non-ergodic action}\label{sec:nonergodic}

In the literature, ratio sets and type are usually only discussed for ergodic actions. The next result determines the ratio set of a non-ergodic action in terms of its ergodic components. Although will use this in Theorem \ref{thm:AvEP}, it is not necessary if one assumes the action $\Gamma \cc (B,\nu)$ is ergodic.

\begin{thm}\label{thm:nonergodic}
Let $\Gamma \cc (B,\nu)$ be a non-singular action on a standard probability space. Then an element $t\in [0,\infty]$ is in the ratio set of $\Gamma \cc (B,\nu)$ if and only if $t$ is in the ratio set of almost every ergodic component of $\Gamma \cc (B,\nu)$.
\end{thm}

Before proving this result, let us clarify the statement. For any $g\in \Gamma$, the Radon-Nikodym derivative $b \mapsto \frac{d\nu\circ g}{d\nu}(b)$ is only defined almost everywhere and, in particular, it only satisfies the cocycle identity almost everywhere. However, by \cite[Theorem B.9]{Zi84}, there exists a Borel cocycle $\alpha:\Gamma \times B \to \RR_{>0}$ which agrees with the Radon-Nikodym derivative $\frac{d\nu \circ g}{d\nu}(b)$ almost everywhere and satisfies the cocycle equation:
$$\alpha(g_2,g_1x)\alpha(g_1,x)=\alpha(g_2g_1,x)$$
for every $g_1,g_2\in \Gamma$ and {\em every} $x\in X$. Let us fix such a cocycle. By the ergodic decomposition theorem in \cite{GS00}, there is a probability measure $\omega$ on the space of ergodic $\Gamma$-quasi-invariant probability measures on $B$ such that 
$$\nu = \int \eta~d\omega(\eta)$$
and for $\omega$-a.e. $\eta$ and every $g\in \Gamma$,
$$\alpha(g,b) =  \frac{d\eta\circ g}{d\eta}(b)$$
for $\eta$-a.e. $b \in B$. Theorem \ref{thm:nonergodic} applies to any ergodic decomposition of this form. For the rest of this section, fix such an $\alpha$ and $\omega$.

\begin{lem}
Let $\Gamma \cc (B,\nu)$ be a non-singular action on a standard probability space. For $0\le t<\infty, \epsilon>0, A \subset B$ and $g\in \Gamma$, let
$$A_{\epsilon,t,g} = \left\{a \in A:~ ga \in A \textrm{ and } \left| \alpha(g,a) - t\right| < \epsilon\right\}$$
and
$$A_t= \cap_{n=1}^\infty \cup_{g\in \Gamma \setminus \{e\}} A^\alpha_{1/n,t,g}.$$
Then $t$ is in ratio set of $\Gamma \cc (B,\nu)$ if and only if $\nu(A_t)  = \nu(A)$ for every positive measure set $A \subset B$.
\end{lem}

\begin{proof}
Suppose $t$ is in the ratio set. Let $A \subset B$ be a set of positive measure. Let 
$$\bar{A}_n = A \setminus \cup_{g\in \Gamma \setminus \{e\}} A_{1/n,t,g}.$$
For any subset $A' \subset \bar{A}_n$ there does not an element $g\in \Gamma \setminus \{e\}$ such that $gA' \subset \bar{A}_n$ and $|\frac{d\nu\circ g}{d\nu}(a) - t| < 1/n$ for a.e. $a\in A'$. Because $t$ is in the ratio set, this implies $\nu(\bar{A}_n)=0$. Thus $\nu(A^\alpha_t)=\nu(A)$. The converse is clear.
\end{proof}

\begin{proof}[Proof of Theorem \ref{thm:nonergodic}]
We will treat the case $t \in [0,\infty)$ only. The case $t=+\infty$ is similar. So suppose $t$ is in the ratio set. Observe that for any $A \subset B$, $A_t \subset A$. So
$$\nu(A) =\int \eta(A)~d\omega(\eta)\ge  \int \eta(A_t)~d\omega(\eta)=\nu(A_t)$$
with equality holding if and only if $\eta(A_t) = \eta(A)$ for $\omega$-a.e. $\eta$.

It follows from the previous lemma that $t$ is in the ratio set if and only if for every measurable $A \subset B$, $\nu(A)=\nu(A_t)$ which, by the equation above, is equivalent to $\eta(A_t)=\eta(A)$ for $\omega$-a.e. $\eta$. By the previous lemma again, this is equivalent to the statement that $t$ is in the ratio set of a.e. ergodic component of $\Gamma \cc (B,\nu)$.

% Let $A \subset B$ be a positive measure subset. By the previous lemma, $\nu(A_t)=\nu(A)$. Therefore, for almost every measure $\nu_0$ in the ergodic decomposition of $\nu$, $\nu_0(A_t)=\nu_0(A)$ which implies that $t$ is in the ratio set of $\Gamma \cc (B,\nu_0)$ as required.

%Conversely, suppose $t$ is in the ratio set of almost every ergodic component of the action. Then for almost every measure $\nu_0$ in the ergodic decomposition of $\nu$, $\nu_0(A_t)=\nu_0(A)$ which implies, by integration, that $\nu(A_t)=\nu(A)$. So $t$ is in the ratio set of $\Gamma \cc (B,\nu)$ as required.
\end{proof}

\begin{cor}
Let $\Gamma \cc (B,\nu)$ be a non-singular action on a standard probability space. Then the ratio set of $\Gamma \cc (B,\nu)$ is either  $\{1\}, \{0,1,\infty\}$, $[0,\infty]$, or $\{0,\infty\} \cup \{\lambda^n:~n\in \ZZ\}$ for some $\lambda \in (0,1)$. Similarly, the stable ratio set of $\Gamma \cc (B,\nu)$ is either $\{1\}, \{0,1,\infty\}$, $[0,\infty]$, or $\{0,\infty\} \cup \{\lambda^n:~n\in \ZZ\}$ for some $\lambda \in (0,1)$. 
\end{cor}
\begin{proof}
As explained in the introduction, this result hold for the ratio set if $\Gamma \cc (B,\nu)$ is ergodic. Theorem \ref{thm:nonergodic} now implies this result holds for the ratio set in general. The statement for stable ratio set is an immediate consequence of the statement for the ratio set (since the stable ratio set is an intersection of ratio sets). 
\end{proof}

\begin{remark}
By convention, the type of a non-ergodic action is undefined and the stable type of a non-weakly-mixing action is undefined. 
\end{remark}

%We will follow the usual convention of speaking of the type of an action only when that action is ergodic. Similarly, we will speak of the stable type only when the action is weakly mixing.

\section{A criterion for stable type $III_\lambda$ with $\lambda>0$}\label{sec:AvEP}

%**define $R(g,b)$ everywhere using Zimmer's appendix (so it satisfies the cocycle equation).

The purpose of this section is to prove a criterion for a non-singular action $\Gamma \cc (B,\nu)$ to be stable type $III_\lambda$ for some $\lambda>0$. First we assume there is a metric $d_B$ on $B$ compatible with its Borel structure such that $(B,d_B)$ is a compact topological space. Let $R:\Gamma \times B \to \RR$ be an additive Borel cocycle such that
$$R(g,b)= \log \frac{d\nu\circ g}{d\nu}(b)$$
$\nu$-almost everywhere. We assume that $R$ satisfies the cocycle equation $R(g_2g_1,b)=R(g_2,g_1b)+R(g_1,b)$ everywhere. Such a cocycle exists by  \cite[Theorem B.9]{Zi84}.

\begin{defn}\label{defn:zeta}
A Borel family of nonnegative Borel functions $\{\Upsilon_n\}_{n=1}^\infty$, $\Upsilon_n: \Gamma \times B \times B \to \RR$ is an {\em admissible family} if, when $S_n:=\{(g,b,b'):~\Upsilon_n(g,b,b')>0\}$,
\begin{enumerate}
\item For every $b,n$, $\sum_{g\in \Gamma} \int \Upsilon_n(g,b,b') ~d\nu(b')=1$. 
\item There is a function $\beta:\NN \to \RR$ such that 
\begin{enumerate}
\item $\lim_{n\to\infty} \beta(n) = 0$;
\item for all $(g,b,b')\in S_n$, $d_B(b,b') \le \beta(n)$ and $d_B(g^{-1} b, g^{-1} b') \le \beta(n)$.
\end{enumerate} 
\item There is a constant $C>0$ such that for almost every $(g,b,b') \in S_n$ (with respect to the product of counting measure on $\Gamma$ and $\nu\times\nu$),
$$|R(g^{-1},b)| + |R(g^{-1},b')|   \le C.$$
\item For some constant $C>0$,
\begin{eqnarray*}
 \int \sum_{g\in \Gamma}  \Upsilon_{n}(g,b,b')~d\nu(b) &\le& C\quad  \forall n \textrm{ and for a.e. } b'\\
 \int \sum_{g\in \Gamma}  \Upsilon_n(g,b,gb') \frac{d\nu\circ g}{d\nu}(b') ~d\nu(b)& \le& C\quad  \forall n \textrm{ and for a.e. } b'\\
   \int \sum_{g\in \Gamma}  \Upsilon_n(g,gb,b') \frac{d\nu\circ g}{d\nu}(b)~d\nu(b')& \le& C\quad  \forall n \textrm{ and for a.e. } b.
    \end{eqnarray*}
    \end{enumerate}
\end{defn}
The functions $\Upsilon_n$ are used as kernels for integral operators in the next subsection. The purpose of this section is to prove:

\begin{thm}\label{thm:AvEP}
%Suppose the action $\Gamma \cc (B,\nu)$ is weakly mixing in the sense that for any pmp ergodic action $\Gamma \cc (X,\mu)$ the diagonal action $\Gamma \cc (B,\nu \times \kappa)$ is ergodic. 
Let $(B,\nu)$ be a standard probability space. Suppose $\Gamma \cc (B,\nu)$ is a non-singular action and there is an admissible family $\{\Upsilon_n\}_{n=1}^\infty$ for $\Gamma \cc (B,\nu)$. Let $\zeta_{n}$ be the probability measure on $\RR$ defined by
$$\zeta_{n}(E) = \sum_{g \in \Gamma} \iint   1_E\left( R(g^{-1},b') - R(g^{-1},b) \right) \Upsilon_n(g,b,b') ~d\nu(b')d\nu(b).$$
Let $\zeta_\infty$ be any weak* limit of $\{\zeta_n\}_{n=1}^\infty$. For every $T$ in the support of $\zeta_\infty$, $e^T$ is in the stable ratio set of $\Gamma \cc (B,\nu)$.  

%Thus if $\zeta_\infty$ is not concentrated on $\{0\}$ then $\Gamma \cc (B,\nu)$ is not type $III_0$.
 %If $T\ne 0$ is in the support of $\zeta_\infty$ then the action $\Gamma \cc (B,\nu)$ is type $III_\lambda$ for some $\lambda \in (0,1]$. Indeed, either $\lambda=1$ or $\lambda =e^{T/l}$ for some integer $l \ne 0$. 
\end{thm}

\begin{remark}
Because $|R(g^{-1},b)| + |R(g^{-1},b')|$ is uniformly bounded on the support of $\Upsilon_n$, it follows that there is a compact interval $[-2C,2C]$ containing the support of each measure $\zeta_n$. Therefore, a weak* limit point of the sequence $\{\zeta_n\}_{n=1}^\infty$ exists by the Banach-Alaoglu Theorem.
\end{remark}

\begin{cor}\label{cor:AvEP}
If the hypotheses of Theorem \ref{thm:AvEP} are satisfied and, in addition, there is a nonzero $T$ in the support of $\zeta_\infty$ and $\Gamma \cc (B,\nu)$ is ergodic then $\Gamma \cc (B,\nu)$ is not type $III_0$. If $\Gamma \cc (B,\nu)$ is weakly mixing then $\Gamma \cc (B,\nu)$ is not stable type $III_0$.
\end{cor}

\begin{remark}
In applying Theorem \ref{thm:AvEP} to the Gromov hyperbolic case, we will construct an admissible family $\{\Upsilon_n\}$ in such a way that the support of $\zeta_n$ is bounded away from zero (i.e., $|R(g^{-1},b') - R(g^{-1},b)|$ is bounded away from zero on the support of $\Upsilon_n$ for every $n$). It follows that the stable ratio set is not contained in $\{0,1,\infty\}$.
\end{remark}

%%%%%%%%%%%%%%%
% this is a little note to myself.
% \int f(x) \frac{d\nu \circ g}{d\nu}(x) ~d\nu(x) = \int f(x) ~d\nu\circ g = (\nu\circ g)(f) = ?
%  For example if f is the characteristic function of a set A then (\nu \circ g)(f)=\nu(gA)
% However f \circ g is the characteristic function of g^{-1}A and f \circ g^{-1} is the characteristic function of gA.
% So, in this case, (\nu \circ g)(f) = \nu(gA) = \nu( f \circ g^{-1}). This is true in general.

\subsection{Operators}\label{sec:operators}

Let $\{\Upsilon_n\}_{n=1}^\infty$ be an admissible family. Let $C>0, R(g,b)$, and so on be as in Definition \ref{defn:zeta}. For $t\in \RR$, let $A_t:\RR \to \RR$ be addition by $t$ (so $A_t(r)=r+t$). Let $\theta$ be a probability measure on $\RR$ equivalent to Lebesgue measure and such that, for some constant $C'$ and every $t_0 \in \RR$ with $|t_0|\le C$,
$$\frac{d \theta \circ A_{t_0}}{d\theta} \le C'$$
almost everywhere. For example, we could choose $\theta$ to satisfy $d\theta := (1/2)e^{-|t|}~dt$. %$d\theta(t) = (2\pi)^{-1/2} e^{-t^2} dt$.

$\Gamma$ acts on $B \times \RR$ by $g(b,t) := (gb,t+R(g,b))$ and on $L^1(\nu\times\theta)$ by $g\cdot f:=f\circ g^{-1}$. Define operators $\sW_n,\sX_n,\sY_n,\sZ_n$ on $L^1(\nu\times\theta)$ by: 
\begin{eqnarray*}
\sW_n f(b,t) &:=&\sum_{g\in \Gamma} \int f(b',t)\Upsilon_n(g,b,b')  ~ d\nu(b')\\
\sX_n f(b,t) &:=& \sum_{g\in \Gamma} \int f(g^{-1}b',t + R(g^{-1},b') )\Upsilon_n(g,b,b')  ~ d\nu(b') \\
\sY_n f(b,t) &:=&\sum_{g\in \Gamma} \int  f(g^{-1}b, t + R(g^{-1},b')  )\Upsilon_n(g,b,b')  ~ d\nu(b')\\
\sZ_n f(b,t) &:=&\sum_{g\in \Gamma} \int  f(b, t + R(g^{-1},b') - R(g^{-1},b))\Upsilon_n(g,b,b')  ~ d\nu(b').
\end{eqnarray*}

The main result of this subsection is:
\begin{prop}\label{prop:1Lambda}
Suppose $\{\Upsilon_n\}_{n=1}^\infty$ is an admissible family. Then for any $\Gamma$-invariant $f \in L^1(\nu\times\theta)$,
$$\lim_{n\to\infty} \| f - \sZ_n f\| = 0.$$
\end{prop}
First we prove that these operators are uniformly bounded.

\begin{prop}\label{prop:1bounded}
There is a constant $C_1>0$ (independent of $n$) such that the operator norms of $\sW_n,\sX_n$ and $\sY_n$ are bounded by $C_1$.
\end{prop}

\begin{proof}
Let $f\in L^1(\nu\times\theta)$ be nonnegative.

{\bf Case $\sW_n$}. Because $\sum_{g\in \Gamma} \int  \Upsilon_{n}(g,b,b')~d\nu(b) \le C$,
\begin{eqnarray*}
\| \sW_n f\| &=&   \iint |\sW_n f| ~d\nu d\theta =  \sum_{g\in \Gamma}\iiint f(b',t) \Upsilon_{n}(g,b,b') ~ d\nu(b')  d\nu(b)d\theta(t) \\
&\le&  C \iint f(b',t)  ~ d\nu(b')d\theta(t) = C \|f\|.
\end{eqnarray*}

 {\bf Case $\sX_n$}. 
Because  $\int \sum_{g\in \Gamma}  \Upsilon_n(g,b,gb') \frac{d\nu\circ g}{d\nu}(b') ~d\nu(b) \le C$,
\begin{eqnarray*}
\| \sX_n f\| &=&   \iint |\sX_n f| ~d\nu d\theta \\
&= &  \sum_{g\in \Gamma} \iiint f(g^{-1}b',t + R(g^{-1},b') ) \Upsilon_n(g,b,b') ~ d\nu(b')  d\nu(b)d\theta(t) \\
&= &  \sum_{g\in \Gamma} \iiint f(g^{-1}b',t  ) \Upsilon_n(g,b,b') \frac{d\theta \circ A_{-R(g^{-1},b')}}{d\theta}(t)~ d\nu(b')  d\nu(b)d\theta(t) \\
&\le &  C' \sum_{g\in \Gamma} \iiint f(g^{-1}b',t  ) \Upsilon_n(g,b,b') ~ d\nu(b')  d\nu(b)d\theta(t) \\
&=&  C' \sum_{g\in \Gamma} \iiint f(b',t  ) \Upsilon_n(g,b,gb') \frac{d\nu\circ g}{d\nu}(b') ~ d\nu(b')  d\nu(b)d\theta(t) \\
&\le& CC' \iint f(b',t  ) ~ d\nu(b') d\theta(t) = CC' \|f \|.
\end{eqnarray*}
%The last equality uses the change of variables $b' \mapsto gb'$, $t\mapsto t-R(g,b')$. The above is bounded by a constant times $\|f\|$ because 

%{\bf XXXXXXXXXXXXXXXX    Is this really the right sign according to definition on p. 35 ????}

%$ \frac{d\nu \circ g^{-1}}{d\nu}(b')  = e^{R(g,b')}$, $|R(g,b')| \le C$ and $\sum_{g\in \Gamma} \int  \Upsilon_n(g,b,gb') ~d\nu(b) \le C$.

{\bf Case $\sY_n$}.
Because $ \int \sum_{g\in \Gamma}  \Upsilon_n(g,gb,b') \frac{d\nu\circ g}{d\nu}(b)~d\nu(b') \le C$,
\begin{eqnarray*}
\| \sY_n f\| &=&   \iint |\sY_n f| ~d\nu d\theta \\
&= & \sum_{g\in \Gamma} \iiint f(g^{-1}b,t+ R(g^{-1},b')) \Upsilon_n(g,b,b') ~ d\nu(b')  d\nu(b)d\theta(t) \\
&= & \sum_{g\in \Gamma} \iiint f(g^{-1}b,t) \Upsilon_n(g,b,b') \frac{d\theta \circ A_{-R(g^{-1},b')}}{d\theta}(t) ~ d\nu(b')  d\nu(b)d\theta(t) \\
&\le &C' \sum_{g\in \Gamma} \iiint f(g^{-1}b,t) \Upsilon_n(g,b,b') ~ d\nu(b')  d\nu(b)d\theta(t) \\
&= &C' \sum_{g\in \Gamma} \iiint f(b,t) \Upsilon_n(g,gb,b') \frac{d\nu\circ g}{d\nu}(b)~ d\nu(b')  d\nu(b)d\theta(t) \\
&\le& CC'  \iint f(b,t) ~  d\nu(b)d\theta(t)  = CC' \| f\|.
\end{eqnarray*}
%The last equality uses the change of variables $b \mapsto gb$, $t\mapsto t-R(g,b')$ and $\frac{d\nu \circ g^{-1}}{d\nu}(b) = e^{ R(g,b)}$. This is bounded by a constant times $\|f\|$ because $|R(g,b)| + |R(g,b')| \le C$ and $\sum_{g\in \Gamma} \int \Upsilon_n(g,gb,b') ~d\nu(b')\le C$.

%{\bf Case $\sZ$}.
%\begin{eqnarray*}
%&&\|\sZ_n f\| = \sum_{j=0}^\infty e^{-2 j} \int |\sZ_n(f)(h,t)| 1_{(-\infty,j)}(t)~d\nu\times\theta(h,t) \\
%&=& \sum_{j=0}^\infty e^{-2 j} \int \sum_{g\in \Gamma} \int  f(h, t + h'(g)- h(g))L_n(g,h,h')e^{ t} 1_{(-\infty,j)}(t)  ~ d\nu(h') d\nu(h)dt \\
%&=& \sum_{j=0}^\infty e^{-2 j} \int \sum_{g\in \Gamma} \int  f(h, t )L_n(g,h,h')\frac{e^{ (t-h'(g)+h(g))}}{} 1_{(-\infty,j)}(t-h'(g)+h(g))  ~ d\nu(h') d\nu(h)dt \\
%&\le& C 10^{2C+1}e^{2 C}\|f\|.
%\end{eqnarray*}

 \end{proof}

\begin{lem}
For every $f \in L^1(\nu\times\theta)$, 
\begin{eqnarray*}
\lim_{n\to\infty} \|f - \sW_n f \|= \lim_{n\to\infty} \| \sX_n f - \sY_n f \|  = 0. 
\end{eqnarray*}
\end{lem}

\begin{proof} %By choosing a topological model for $B$, we may assume $B$ is a compact metrizable space and the action $\Gamma \cc B$ is by homeomorphisms. 
Suppose that $f$ is a continuous function on $B \times \RR$ with compact support. Because $\{\Upsilon_n\}_{n=1}^\infty$ is admissible, if $(g,b,b')$ is such that $\Upsilon_n(g,b,b') >0$ then $d_B(b,b') \le \beta(n)$ where $\lim_{n\to\infty} \beta(n) = 0$. Because $\Upsilon_n$ is a probability density, $\sW_n f $ converges to $f$ uniformly on compact sets. So the bounded convergence theorem implies $\lim_{n\to\infty} \|f - \sW_n f \|= 0$. 

Observe that
\begin{eqnarray*}
&&(\sX_n f - \sY_n f)(b,t)\\
 &=& \sum_{g\in \Gamma}  \int \big[f(g^{-1}b' ,t+R(g^{-1},b')) - f(g^{-1}b, t + R(g^{-1},b'))\big]\Upsilon_n(g,b,b') ~d\nu(b').
% &=& \sum_{g\in \Gamma} \int \big[f(g^{-1} b' ,t +R(g,b')) - f(g^{-1} b, t + R(g,b'))\big]\Upsilon_n(g,b,b') ~d\nu(b')\\
 \end{eqnarray*}
Because $d_B(g^{-1} b', g^{-1} b) \le \beta(n)$, uniform continuity of $f$ implies $\sX_nf - \sY_nf$ converges to zero pointwise and uniformly on compact sets. So the bounded convergence theorem implies $\lim_{n\to\infty}\| \sX_n f - \sY_n f \|= 0$. 

Since compactly supported continuous functions are dense in the norm topology on $L^1(\nu\times\theta)$ and the operators $\sW_n,\sX_n,\sY_n$ are uniformly bounded (by the previous proposition) the lemma follows.
\end{proof}
We can now prove Proposition \ref{prop:1Lambda} (which states that for any $\Gamma$-invariant $f \in L^1(\nu\times\theta)$, $\lim_{n\to\infty} \| f - \sZ_n f\| = 0$).

\begin{proof}[Proof of Proposition \ref{prop:1Lambda}]
Because $f$ is $\Gamma$-invariant, 
\begin{eqnarray*}
\sX_n f(b,t) &=& \sum_{g\in \Gamma} \int (g \cdot f)( b' ,t )\Upsilon_n(g,b,b') ~d\nu(b')\\
&=& \sum_{g\in \Gamma} \int f( b' ,t )\Upsilon_n(g,b,b') ~d\nu(b')= \sW_n f(b,t).
\end{eqnarray*}
Also,
\begin{eqnarray*}
\sY_n f(b,t) &=&  \sum_{g\in \Gamma} \int  (g \cdot f)(b, t + R(g^{-1},b')-R(g^{-1},b)) \Upsilon_n(g,b,b') ~d\nu(b')\\
&=&  \sum_{g\in \Gamma} \int  f(b, t + R(g^{-1},b')-R(g^{-1},b)) \Upsilon_n(g,b,b') ~d\nu(b')= \sZ_n f(b,t).
\end{eqnarray*}
The previous lemma now implies
\begin{eqnarray*}
0 &=& \lim_{n\to\infty}  \| \sX_n f - \sY_n f \|  =  \lim_{n\to\infty}  \| \sW_n f - \sZ_n f \| \\
&\ge&  \lim_{n\to\infty}  \|  f - \sZ_n f \| - \|f - \sW_n f\| = \lim_{n\to\infty}  \|  f - \sZ_n f \|.
\end{eqnarray*}

%The proposition is now implied by the previous lemma.
\end{proof}

\subsection{Proof of Theorem \ref{thm:AvEP}}\label{sec:end}

\begin{lem}
Suppose $T \in \RR$ is such that $e^T$ is not in the ratio set of $\Gamma \cc (B,\nu)$. Then there exists an $\epsilon>0$ and a $\Gamma$-invariant, positive measure set $A \subset B\times \RR$ such that for every $(b,t) \in A$ and every $t' \in (-\epsilon,\epsilon)$, $(b,t+T+t') \notin A$. 
\end{lem}

\begin{proof}
Because $e^T$ is not in the ratio set, there exists a positive measure set $B' \subset B$ and an $\epsilon>0$ such that for every $b\in B'$ if $g\in \Gamma$ is such that $gb \in B'$ then 
$$R(g,b) \notin (T-3\epsilon,T+3\epsilon).$$
%By choosing a version of the Radon-Nikodym derivative, we may assume the above holds for every $b \in B'$.

Let $A = \Gamma (B'\times (-\epsilon,\epsilon)) \subset B \times \RR$. It is clear that $A$ is $\Gamma$-invariant and has positive measure. Every element of $A$ has the form $\gamma(b,t)=(\gamma b, t+R(\gamma,b))$ for some $\gamma \in \Gamma, b \in B', t\in (-\epsilon,\epsilon)$. Suppose, to obtain a contradiction, that $(\gamma b, t+R(\gamma,b) + T + t') \in A$ for some $t' \in (-\epsilon,\epsilon)$. Then there exists $g \in \Gamma$, $b' \in B'$ and $t'' \in (-\epsilon,\epsilon)$ such that
$$(\gamma b, t+R(\gamma,b) + T + t')  = g(b',t'').$$
We multiply both sides on the left by $\gamma^{-1}$ to obtain
$$(b,t+T + t') = \gamma^{-1}g(b',t'') = (\gamma^{-1}gb', t'' + R(\gamma^{-1}g, b')).$$
Therefore $R(\gamma^{-1}g, b') = t+ T + t' -  t'' \in (T-3\epsilon,T+3\epsilon).$ 
This contradicts the choice of $A$. So $A$ satisfies the conclusion as required.
\end{proof}

The next lemma is Theorem \ref{thm:AvEP} with ``stable ratio set'' replaced with ``ratio set''.
\begin{lem}\label{lem:AvEP}
%Suppose the action $\Gamma \cc (B,\nu)$ is weakly mixing in the sense that for any pmp ergodic action $\Gamma \cc (X,\mu)$ the diagonal action $\Gamma \cc (B,\nu \times \kappa)$ is ergodic. 
Let $(B,\nu)$ be a standard probability space. Suppose $\Gamma \cc (B,\nu)$ is a non-singular action and there is an admissible family $\{\Upsilon_n\}_{n=1}^\infty$ for $\Gamma \cc (B,\nu)$. Let $\zeta_{n}$ be the probability measure on $\RR$ defined by
$$\zeta_{n}(E) = \sum_{g \in \Gamma} \iint   1_E\left( R(g^{-1},b') - R(g^{-1},b) \right) \Upsilon_n(g,b,b') ~d\nu(b')d\nu(b).$$
Let $\zeta_\infty$ be any weak* limit of $\{\zeta_n\}_{n=1}^\infty$. Then for every $T$ in the support of $\zeta_\infty$, $e^T$ is in the ratio set of $\Gamma \cc (B,\nu)$.  

%Thus if $\zeta_\infty$ is not concentrated on $\{0\}$ then $\Gamma \cc (B,\nu)$ is not type $III_0$.
 %If $T\ne 0$ is in the support of $\zeta_\infty$ then the action $\Gamma \cc (B,\nu)$ is type $III_\lambda$ for some $\lambda \in (0,1]$. Indeed, either $\lambda=1$ or $\lambda =e^{T/l}$ for some integer $l \ne 0$. 
\end{lem}

\begin{proof}
Let $T$ be an element of the support of $\zeta_\infty$. To obtain a contradiction, suppose that the ratio set of $\Gamma \cc (B,\nu)$ does not contain $e^T$. 

Let $A \subset B\times \RR$ and $\epsilon>0$ be as in the previous lemma. Let $f$ be the characteristic function of $A$. Note
\begin{eqnarray*}
\sZ_nf (b,t) &=&\sum_{g\in \Gamma} \int  f(b, t + R(g^{-1},b') - R(g^{-1},b))\Upsilon_n(g,b,b')  ~ d\nu(b')=  \int f(b,t+t') ~d\zeta_{n,b}(t')
\end{eqnarray*}
where $\zeta_{n,b}$ is the probability measure on $\RR$ given by
$$\zeta_{n,b}(E) = \sum_{g \in \Gamma} \int   1_E\left( R(g^{-1},b') - R(g^{-1},b) \right) \Upsilon_n(g,b,b') ~d\nu(b').$$
By Proposition \ref{prop:1Lambda}, 
\begin{eqnarray*}
0 &=& \lim_{n\to\infty} \| f - \sZ_n f\| = \lim_{n\to\infty} \int \left| f(b,t) -  \int f(b,t+t') ~d\zeta_{n,b}(t') \right| ~d\nu(b)d\theta(t)\\
&\ge&  \lim_{n\to\infty} \int_A \left| f(b,t) -  \int f(b,t+t') ~d\zeta_{n,b}(t') \right| ~d\nu(b)d\theta(t)\\
&\ge&  \lim_{n\to\infty} \int_A  \zeta_{n,b}((T-\epsilon,T+\epsilon)) ~d\nu(b)d\theta(t)\\
&=& \lim_{n\to\infty} \nu\times \theta(A)\zeta_n((T-\epsilon,T+\epsilon)).
\end{eqnarray*}
The second inequality holds because, by the previous lemma, if $(b,t)\in A$ and $t' \in (T-\epsilon,T+\epsilon)$ then $(b,t+t') \notin A$ so $f(b,t)-f(b,t+t')=1$. The last equality holds because $\int \zeta_{n,b} ~d\nu(b) = \zeta_n$. Because $0<\nu\times\theta(A)$, $T$ must not be in the support of $\zeta_\infty$. This contradiction implies the lemma.
\end{proof}

The strategy for proving Theorem \ref{thm:AvEP} from Lemma \ref{lem:AvEP} is to show that, given any pmp action $\Gamma \cc (K,\kappa)$ there exists a topological model for this action and an admissible family $\{\Upsilon'_n\}_{n=1}^\infty$ for the product action $\Gamma \cc (B\times K,\nu\times\kappa)$ so that if $T$ is in the support of $\zeta_\infty$, then $T$ is also in the support of $\zeta'_\infty$ where $\zeta'_\infty$ is the measure corresponding to $\{\Upsilon'_n\}_{n=1}^\infty$. It will be helpful to know that $\Gamma \cc (K,\kappa)$ has a particularly nice topological model; which is the import of the next result.  

\begin{prop}\label{prop:model}
Let $\Gamma \cc (X,\mu)$ be an ergodic pmp action. Then there exists a compact metric space $(K,d_K)$ with a Borel probability measure $\kappa$ and a continuous action $\Gamma \cc K$ such that
\begin{itemize}
\item $\Gamma \cc (X,\mu)$ is measurably conjugate to $\Gamma \cc (K,\kappa)$
\item for every $\epsilon>0$ and $x,y \in K$,
$$1/3 \le \frac{\kappa(B(x,\epsilon))}{\kappa(B(y,\epsilon))} \le 3$$
where for example, $B(x,\epsilon)=\{z\in K:~d_K(x,z)\le \epsilon\}$.
\end{itemize}
\end{prop}

\begin{lem}
Let $(X,\mu)$ be a finite measure space, $n \in \N$, $\mu(X)\ge 2^{-n-1}$ and $\beta$ a finite Borel partition of $X$ such that for every $B \in \beta$, $\mu(B) \le 2^{-n-1}$. Then there exists a partition $\alpha$ with $\alpha \le \beta$ and $|\mu(A) - 2^{-n}| \le 2^{-n-1}$ for every $A \in \alpha$.
\end{lem}

\begin{proof}
We prove this by induction on the cardinality of $\beta$. If $|\beta|=1$ then set $\alpha=\beta$. If $|\mu(X) - 2^{-n}| \le 2^{-n-1}$ then we may set $\alpha$ equal to the trivial partition. So let us assume that $\mu(X)> 2^{-n} + 2^{-n-1}$ and $|\beta| \ge 2$. Then there exists a set $A \subset X$ which is a union of atoms of $\beta$ such that $2^{-n-1} \le \mu(A) \le 2^{-n}$.  Then $\mu(X \setminus A) \ge 2^{-n-1}$. So the induction hypothesis, applied to $X\setminus A$ and $\beta' =\{B\in \beta:~ B \cap A = \emptyset\}$ implies the existence of a partition $\alpha' \le \beta'$ such that $|\mu(A') - 2^{-n}| \le 2^{-n-1}$ for every $A' \in \alpha'$. Set $\alpha= \alpha' \cup \{A\}$ to finish the lemma.
\end{proof}

\begin{proof}[Proof of Proposition \ref{prop:model}]
If $(X,\mu)$ is purely atomic then because the action is ergodic, there exist elements $x_1,\ldots, x_n \in X$ such that $\mu(\{x_i\})=1/n$ for each $i$. In this case, we may let $K=\{1,\ldots,n\}$, $d_K(i,j)=1$ if $i\ne j$ and $\Gamma$ acts on $K$ by $g \cdot i = j \Leftrightarrow gx_i=x_j$. So we may assume $(X,\mu)$ is purely non-atomic. 

It is well-known that there exists a topological model for the action. Briefly, this is achieved by constructing a $\Gamma$-invariant separable subalgebra $\cA$ of $L^\infty(X,\mu)$ which is dense in $L^2(X,\mu)$ and then letting $\hat{X}$ denote the maximal ideal space of $\cA$. This procedure produces  a fully supported Borel probability measure $\kappa$ on the Cantor set, which we realize as the product space $\{0,1\}^\N$, and a continuous action $\Gamma \cc \{0,1\}^\N$ such that $\Gamma \cc (\{0,1\}^\N,\kappa)$ is measurably conjugate with $\Gamma \cc (X,\mu)$. 

For $k\ge 1$, let $\beta_k$ be the partition of $\{0,1\}^\N$ determined by the condition: for any $x,y \in \{0,1\}^\N$, $x$ and $y$ are in the same partition element of $\beta_k$ if and only if $x_i=y_i$ for all $1\le i \le k$. Also, let $\beta_0$ be the trivial partition.

We claim that there exist sequences $\{n_k\}_{k=0}^\infty, \{m_k\}_{k=0}^\infty$ of nonnegative integers and finite partitions $\{\alpha_{k}\}_{k=0}^\infty$ satisfying $(\forall k \ge 1)$
\begin{enumerate}
\item $n_{k+1} >n_k$, $m_{k+1} > m_k$;
%\item $\alpha_k\le \alpha_{k+1}$;
\item $\beta_{k} \le \alpha_{k} \le \alpha_{k+1} \le \beta_{m_{k+1}}$;
\item for every $A \in \alpha_k$, $|\kappa(A) - 2^{-n_k}| \le 2^{-n_k-1}$;
\end{enumerate}
We will prove this by induction on $k$. The base case is handled by setting $n_0=m_0=0$ and $\alpha_0$ equal to the trivial partition. So suppose that $n_0,\ldots, n_k, m_0,\ldots, m_{k}$ and $\alpha_0, \ldots, \alpha_k$ have been chosen satisfying the above. Although $m_{k+1}$ has not been chosen at this stage, we do require that $\alpha_{k}$ is refined by $\beta_p$ for some $p$. 

Let $n_{k+1}>n_k$ be an integer so that for every $A \in \alpha_{k} \vee \beta_{k+1}$, $\kappa(A) \ge 2^{-n_{k+1}-1}$ (this exists because $\kappa$ is fully supported). Let $m_{k+1}>m_k$ be a sufficiently large integer so that $\beta_{m_{k+1}}$ refines $\alpha_{k}$ and every $B \in \beta_{m_{k+1}}$ satisfies $\kappa(B)\le 2^{-n_{k+1}-1}$. By applying the previous lemma to each atom $X$ of $\alpha_{k} \vee \beta_{k+1}$ and the restriction of $\beta_{m_{k+1}}$ to $X$, we see that there exists a partition $\alpha_{k+1}$ satisfying
\begin{itemize}
\item $\alpha_{k} \vee \beta_{k+1} \le \alpha_{k+1} \le \beta_{m_{k+1}}$ 
\item for every $A \in \alpha_{k+1}$, $|\kappa(A) - 2^{-n_{k+1}}| \le 2^{-n_{k+1} -1}$.
\end{itemize}
This establishes the claim.

Now set $K=\{0,1\}^\N$ and $d_K(x,y) = \frac{1}{k+1}$ where $k\ge 0$ is the largest integer such that $x$ and $y$ are in the same atom of $\alpha_k$. Because $\beta_k \le \alpha_k$, it follows that $\bigvee_{k=1}^\infty \alpha_k$ is the partition into points, which implies $d_K$ is a metric (instead of a pseudo-metric). Because $\alpha_k \le \beta_{m_{k+1}}$, it follows that each atom of $\alpha_k$ is clopen and therefore $d_K$ is continuous with respect to the product topology on $\{0,1\}^\N$. Finally, we note that for any $x \in K$ and any $\epsilon>0$ if $k$ is the smallest nonnegative integer such that $\frac{1}{k+1} \le \epsilon$ then $B(x,\epsilon)=B(x,\frac{1}{k+1}) \in \alpha_k$. Therefore 
$$|\kappa(B(x,\epsilon)) - 2^{-n_k}| \le 2^{-n_k-1}.$$
So if $y\in K$ is any other point then
$$1/3 \le \frac{\kappa(B(x,\epsilon))}{\kappa(B(y,\epsilon))} \le 3.$$
\end{proof}

\begin{proof}[Proof of Theorem \ref{thm:AvEP}]

By Theorem \ref{thm:nonergodic} it suffices to show that $e^T$ is in the ratio set of $\Gamma \cc (B\times K, \nu\times \kappa)$ whenever $\Gamma \cc (K,\kappa)$ is an {\em ergodic} pmp action. So let $\Gamma \cc (K,\kappa)$ be an ergodic pmp action. By Proposition \ref{prop:model}, we may assume that $(K,d_K)$ is a compact metric space such that for every $\epsilon>0$ and $x,y \in K$,
$$1/3 \le \frac{\kappa(B(x,\epsilon))}{\kappa(B(y,\epsilon))} \le 3.$$

%For every positive integer $n\ge 1$, let $D_n=\{ (k_1,k_2) \in K \times K:~ d_K(k_1,k_2)\le 1/n\}$.  Because $\Gamma$ acts continuously, for every integer $n>0$ and $g\in \Gamma$ there is an $N=N(g)$ such that $\{ (g^{-1}x,g^{-1}y):~ (x,y) \in D_N\} \subset D_n$. Set $D_{n,g} := D_N$. Note that for any $k_1,k_2 \in K$,
%$$1/3\le \frac{\int 1_{D_{n,g}}(k,k_2) ~d\kappa(k)}{ \int 1_{D_{n,g}}(k_1,k) ~d\kappa(k) }  = \frac{ \kappa(B(k_2,1/N))}{\kappa(B(k_1,1/N))}\le 3.$$

Given an integer $n\ge 1$ and $g\in \Gamma$, let $0<\rho(n,g)<1/n$ be such that for every $x,y \in K$ with $d_K(x,y)\le \rho(n,g)$, $d_K(g^{-1}x,g^{-1}y) \le 1/n$. 

Define $\Upsilon'_n:\Gamma \times B\times K \times B \times K \to \RR$ by 
$$\Upsilon'_n(g,b,k,b',k') := \frac{1_{B( k,\rho(n,g))}(k')\Upsilon(g,b,b')}{ \kappa(B(k,\rho(n,g))) }.$$
It is an easy exercise using the above estimates to check that $\{\Upsilon'_n\}_{n=1}^\infty$ is an admissible family for $G \cc (B\times K,\nu \times \kappa)$ with $d_{B\times K}$, a metric on $B\times K$, given by $d_{B\times K}((b,k),(b',k'))=d_B(b,b')+ d_K(k,k')$. Because $\Gamma \cc (K,\kappa)$ is measure-preserving,
$$R(g,b,k):= \log \frac{d(\nu\times \kappa)\circ g}{d(\nu\times \kappa)}(b,k) = R(g,b).$$
So for any $E \subset \RR$,
\begin{eqnarray*}
\zeta_{n}(E) &=& \sum_{g \in \Gamma} \iint   1_E\left( R(g^{-1},b') - R(g^{-1},b) \right) \Upsilon_n(g,b,b') ~d\nu(b')d\nu(b)\\
&=&  \sum_{g \in \Gamma} \iint   1_E\left( R(g^{-1},b',k') - R(g^{-1},b,k) \right) \Upsilon'_n(g,b,k,b',k') ~d\nu\times \kappa(b',k')d\nu \times \kappa(b,k).
\end{eqnarray*}
So Lemma \ref{lem:AvEP} above implies the ratio set of the action $\Gamma \cc (B \times K,\nu\times \kappa)$ contains $e^T$.  Since $\Gamma  \cc (K,\kappa)$ is arbitrary, this proves the result.

\end{proof}

\section{Gromov hyperbolic spaces}\label{sec:Gromov}  %%%%%%%%%%GROMOV

The purpose of this section is to set notation and review Gromov hyperbolic spaces.%A finitely generated group $\Gamma$ is Gromov hyperbolic if it is $\delta$-hyperbolic for some $\delta>0$ with respect to the word metric defined by a finite symmetric generating set. 

Let $(X,d_X)$ be a metric space. The {\em Gromov product} of $x,y \in X$ relative to $z\in X$ is
$$(x|y)_z := \frac{1}{2}\Big( d_X(x,z) + d_X(y,z) - d_X(x,y)\Big).$$
For $\delta>0$, the space $(X,d_X)$ is {\em $\delta$-hyperbolic} if 
\begin{eqnarray}\label{eqn:gromov}
(x|y)_w \ge \min\{ (x|z)_w, (y|z)_w\} - \delta,\quad  \forall x,y,w,z \in X.
\end{eqnarray}
%$\delta$ is a {\em hyperbolicity constant}. 

\subsection{The Gromov boundary}

Let $(X,d_X)$ be a $\delta$-hyperbolic space. A sequence $\{x_i\}_{i=1}^\infty$ in $X$ is a {\em Gromov sequence} if
$$\lim_{i,j \to \infty} (x_i|x_j)_z = +\infty$$
for some (and hence, any) basepoint $z\in X$. Two Gromov sequences $\{x_i\}_{i=1}^\infty$, $\{y_i\}_{i=1}^\infty $ are {\em equivalent} if $\lim_{i\to\infty} (x_i|y_i)_z = +\infty$ with respect to some (and hence any) basepoint $z$. It is an exercise to show that this defines an equivalence relation (assuming $(X,d_X)$ is $\delta$-hyperbolic). The {\em Gromov boundary} is the space of equivalence classes of Gromov sequences. We denote it by $\partial X$.  Let $\overline{X}$ denote $X \cup \partial X$. %If $\{x_i\}_{i=1}^\infty$ is a Gromov sequence in $X$ and $\xi \in \partial X$, we write $\lim_{i\to\infty} x_i = \xi$ if $\{x_i\}_{i=1}^\infty \in \xi$. 

The Gromov product extends to $\partial X$ as follows. Let $p,z \in X$ and $\xi, \eta \in \partial X$. Define
$$(\xi|p)_z := \inf \liminf_{i \to \infty} (x_i|p)_z, \quad (\xi|\eta)_z := \inf \liminf_{i \to \infty} (x_i|y_i)_z$$
where the infimums are over all sequences $\{x_i\}_{i=1}^\infty \in \xi, \{y_i\}_{i=1}^\infty\in \eta$. By \cite[ Lemma 5.11]{Va05}
\begin{eqnarray}\label{eqn:gromovproduct}
\limsup_{i \to \infty} (x_i|y_i)_z -2\delta \le (\xi|\eta)_z \le \liminf_{i \to \infty} (x_i|y_i)_z
\end{eqnarray}
for any sequences $\{x_i\}_{i=1}^\infty \in \xi, \{y_i\}_{i=1}^\infty\in \eta$. These inequalities also hold if $\eta = p \in X$ and $y_i$ is any sequence with $\lim_{i\to\infty} y_i=p$. According to \cite[Proposition 5.12]{Va05},  inequality (\ref{eqn:gromov}) extends to $x,y \in \partial X$.

In \cite{BH99} it is shown that if $\epsilon>0$ is sufficiently small and $\bar{d}_\epsilon:\overline{X} \times \overline{X} \to \RR$ is defined by
$$\bar{d}_\epsilon(\xi,\eta):=e^{-\epsilon (\xi|\eta)_z}$$
then there exists a metric $\bar{d}$ on $\overline{X}$ and constants $A,B>0$ such that $A\bar{d}_\epsilon \le \bar{d} \le B\bar{d}_\epsilon.$ Any such metric is called a {\em visual metric}. %We write $\overline{X \cup \partial X}$ for the disjoint union of $X$ with $\partial X$ with the topology determined by: (i) the inclusion $X \subset X \cup \partial X$ is a homeomorphism onto its image, (ii) $\partial X$ is closed in $X$, (iii) the inclusion $\partial X \subset X \cup \partial X$ is a homeomorphism onto its image (where $\partial X$ has the topology induced by the metric $\rho$), (iv) for any sequence $\{x_i\}_{i=1}^\infty$ in $X$, $\lim_{i\to\infty} x_i = \xi \in \partial X$ if and only if $\{x_i\}_{i=1}^\infty$ is a Gromov sequence and $\{x_i\}_{i=1}^\infty \in \xi$. The space $\overline{X \cup \partial X}$ is the {\em Gromov compactification} of $X$. (**I don't quite like this way of defining the topology, let's check with the asympotic invariance section where it is used**).

%The next lemma is in [AS90, lemma 4.6 (4)].
%\begin{lem}\label{lem:gromovproduct}
%If $x_i,y_i \in [X \cup \partial X]$, $\lim_{i\to\infty} x_i = x_\infty$ and $\lim_{i\to\infty} y_i = y_\infty$ then for any $z\in X$,
%$$(x_\infty,y_\infty)_z \le \liminf_{i\to\infty} (x_i|y_i)_z \le (x_\infty,y_\infty)_z+2\delta.$$
%\end{lem}

\subsection{Quasi-geodesics and quasi-isometries}\label{sec:defn}

\begin{defn}
Let $(X,d_X)$, $(Y,d_Y)$ be metric spaces. For $\lambda \ge 1$ and $c\ge 0$, a map $\phi:X \to Y$ is a {\em $(\lambda,c)$-quasi-isometric embedding} if for all $x,y \in X$,
$$\lambda^{-1} d_X(x,y) - c \le d_Y(\phi(x),\phi(y)) \le \lambda d_X(x,y) + c.$$
%If, in addition, for every $y\in Y$ there is a $x \in X$ such that $d_Y(\phi(x),y) \le c$ then $\phi$ is a {\em $(\lambda,c)$-quasi-isometry}.
\end{defn}

\begin{defn}
Let $(X,d_X)$ be a Gromov hyperbolic space. A $(\lambda,c)$-quasi-isometric embedding $q$ of an interval $I \subset \RR$ into $X$  is called a {\em $(\lambda,c)$-quasi-geodesic}. If $I$ is a finite interval and its endpoints are mapped to $x,y \in X$ respectively, then we say it is a quasi-geodesic {\em from $x$ to $y$}. If $I=(-\infty, \infty)$ and $\lim_{t \to -\infty} q(t) = \xi_-, 
\lim_{t \to +\infty} q(t) = \xi_+$, then we say $q$ is a quasi-geodesic from $\xi_-$ to $\xi_+$. A similar definition holds for half-infinite intervals.

The space $(X,d_X)$ is {\em $(\lambda,c)$-quasi-geodesic} if for every pair of points $x,y \in X\cup \partial X$ there exists a $(\lambda,c)$-quasi-geodesic from $x$ to $y$.  By abuse of notation, we sometimes identify a quasi-geodesic with its image. For example, it is convenient to denote by $[x,y]$ a $(\lambda,c)$-quasi-geodesic from $x$ to $y$ ($x,y \in X$).
\end{defn}

%We will say that $(X,d_X)$ is {\em uniformly quasi-geodesic} if it is $(1,c)$-quasi-geodesic for some $c>0$. 

%\subsection{Hyperbolic groups}\label{sec:ba}

\begin{defn}[Hyperbolic groups]
$(\Gamma, d)$ is a {\em hyperbolic group} if $d$ is a left-invariant metric on $\Gamma$, $\Gamma$ is a countable discrete group, and $(\Gamma,d)$ is a proper $\delta$-hyperbolic metric space for some $\delta>0$. It is {\em uniformly quasi-geodesic} if it is $(1,c)$-quasi-geodesic for some $c>0$. It is {\em non-elementary} if it is not a finite extension of a cyclic group. 
\end{defn}

\subsection{Quasi-conformal measures and horofunctions}

Let $(X,d_X)$ be a $\delta$-hyperbolic metric space. Choose a basepoint $x_0\in X$. 
\begin{lem}
Let $\xi \in \partial X$ and suppose $\{y_i\}, \{z_i\} \subset X$ are two sequences converging to $\xi$ (w.r.t. the topology on $\overline{X}$). Then for any $w \in X$,
$$\limsup_{i\to\infty} \left| d_X(y_i,w) - d_X(y_i,x_0) - \Big( d_X(z_i,w) - d_X(z_i,x_0) \Big) \right| \le 4\delta.$$
\end{lem}
\begin{proof}
Observe that
$$d_X(y_i,w) - d_X(y_i,x_0)  = d_X(w,x_0) - 2(y_i| w)_{x_0}.$$
A similar statement holds for $z_i$ in place of $y_i$. Thus,
$$\left| d_X(y_i,w) - d_X(y_i,x_0) - \Big( d_X(z_i,w) - d_X(z_i,x_0) \Big) \right|= 2\left| (y_i|w)_{x_0} - (z_i|w)_{x_0} \right|.$$
The lemma now follows from (\ref{eqn:gromovproduct}).
\end{proof}
For $\xi \in \partial X$, define $h_\xi:X \to \RR$ by
$$h_\xi(z): = \inf \liminf_{n\to\infty} d_X(z,y_i) - d_X(y_i,x_0)$$
where the infimum is over all sequences $\{y_i\} \subset X$ which converges to $\xi$. This is the {\em horofunction} associated to $\xi$ (and the basepoint $x_0$). By the previous lemma, if $\{y_i\}$ is any sequence converging to $\xi$ and $z \in X$ is arbitrary then
\begin{eqnarray}\label{eqn:horo}
\limsup_{i\to\infty} \left|h_\xi(z) - \Big(d_X(z,y_i) - d_X(y_i,x_0)\Big) \right| \le 4 \delta.
\end{eqnarray}

\begin{defn}[Quasi-conformal measure]\label{defn:qc}
Suppose $(\Gamma,d)$ is a Gromov hyperbolic group. A Borel probability measure $\nu$ on $\partial \Gamma$ is {\em quasi-conformal} if there are constants $\fh,C>0$ such that for any $g\in \Gamma$ and a.e. $\xi \in \partial \Gamma$,
$$C^{-1} \exp(-\fh h_\xi(g^{-1})) \le \frac{d\nu  \circ g}{d\nu}(\xi) \le C \exp(-\fh h_\xi(g^{-1})).$$
We will call $\fh>0$ the {\em quasi-conformal constant} associated to $\nu$.
\end{defn}

It is well-known that if $d$ comes from a word metric on $\Gamma$ (or more generally, any geodesic metric) then there is a quasi-conformal measure on $\partial \Gamma$ \cite{Co93}. More generally: %By we have:

\begin{lem}\label{lem:conformal}
Let $(\Gamma,d)$ be a non-elementary, uniformly quasi-geodesic, hyperbolic group. Then there exists a quasi-conformal measure $\nu$ on $\partial \Gamma$. Moreover, any two quasi-conformal measures are equivalent. Also if $\fh$ is the quasi-conformal constant of $\nu$ then there is a constant $C>0$ such that
\begin{enumerate}
\item If $B(g,r)$ denotes the ball of radius $r$ centered at $g\in \Gamma$ then 
$$C^{-1} e^{ \fh r} \le |B(g,r)| \le C e^{\fh r},\quad \forall g\in \Gamma, r>0.$$

\item $$C^{-1}e^{-\fh n} \le \nu\left(\{ \xi' \in \partial \Gamma:~ (\xi|\xi')_e \ge n \}\right) \le C e^{-\fh n}, \quad \forall n>0, \xi \in \partial \Gamma.$$
\end{enumerate}
%Then there exists a quasi-conformal measure $\nu$ on $\partial \Gamma$. %Moreover, the action $\Gamma \cc (\cH,\nu)$ is essentially free.
\end{lem}
\begin{proof}
This follows immediately from \cite[Theorem 2.3]{BHM11} and the fact that any non-elementary uniformly quasi-geodesic hyperbolic group is a proper quasi-ruled hyperbolic space by Lemma \ref{lem:quasi-ruled} below.
\end{proof}

The paper \cite{BHM11} contains many results for hyperbolic spaces under the assumption that these spaces are {\em quasi-ruled}. To be precise a metric space $(X,d_X)$ is quasi-ruled if  there are  constants $(\tau,\lambda,c)$ such that $(X,d_X)$ is $(\lambda,c)$-quasi-geodesic and for any $(\lambda,c)$-quasi-geodesic $\gamma: [a,b] \to X$ and any $a \le s \le t \le u \le b$,
$$d_X(\gamma(s),\gamma(t)) + d_X(\gamma(t),\gamma(u)) - d_X(\gamma(s),\gamma(u)) \le 2\tau.$$
\begin{lem}\label{lem:quasi-ruled}
If $(X,d_X)$ is $(1,c)$-quasi-geodesic then it is quasi-ruled.
\end{lem}
\begin{proof}
If $\gamma:[a,b] \to X$ is any $(1,c)$-quasi-geodesic then for any $a \le s \le t \le u \le b$,
$$d_X(\gamma(s),\gamma(t)) + d_X(\gamma(t),\gamma(u)) - d_X(\gamma(s),\gamma(u)) \le 3c.$$
So we may set $\tau=3c/2$.
\end{proof}

\section{The stable ratio set of the Gromov boundary}\label{sec:automatic} %AUTOMATIC ERGODICITY

For the rest of the paper, we assume $(\Gamma,d)$ is a non-elementary, uniformly quasi-geodesic, hyperbolic group and $\nu$ is a quasi-conformal measure on $\partial \Gamma$ with quasi-conformal constant $\fh>0$.

Theorem \ref{thm:main} is an immediate consequence of Theorem \ref{thm:AvEP}, Corollary \ref{cor:AvEP} and the next lemma.  
\begin{lem}\label{lem:admissible}
Let $Q>0$ be a constant. Then there exists an admissible family $\{\Upsilon_n\}_{n=1}^\infty$ for $\Gamma \cc (\partial \Gamma, \nu)$ such that if $\Upsilon_n(g,\xi,\xi')>0$ then $|R(g^{-1},\xi)-R(g^{-1},\xi')| \ge Q$, where 
$$R(g,\xi) =  \log \frac{d\nu\circ g}{d\nu}(\xi).$$
\end{lem}

%\begin{proof}[Proof of Theorem \ref{thm:main}]%
%Because all quasi-conformal measures on $\partial \Gamma$ are equivalent (see e.g., \cite{BHM11} Theorem 2.3), it suffices to prove the statement for $\Gamma \cc (\partial \Gamma, \nu)$. This is implied immediately by Lemma \ref{lem:admissible}, Theorem \ref{thm:AvEP} and Corollary \ref{cor:AvEP}.
%end{proof}

%Let $\nu$ be a conformal probability measure on $\cH$ as  in  \S \ref{sec:measures}. 
%We now turn towards the proof of Lemma \ref{lem:admissible}. 

\begin{defn}
Let $\rho>0$ be a constant to be chosen later. For $\xi \in \partial \Gamma$, let $Y_n(\xi)$ be the set of all $g\in \Gamma$ such that $d(g,e) \in (2n-2\rho, 2n)$ and $h_\xi(g) \in (2\rho, 4\rho)$. For $g\in \Gamma$, let $Z_n(g)$ be the set of all $\xi'\in \partial \Gamma$ such that $(\xi' | g)_e  \in (n, n+\rho)$. Define $\Upsilon_n:\Gamma \times \partial \Gamma \times \partial \Gamma \to \RR$ by
$$\Upsilon_n(g,\xi,\xi') := \frac{1_{Y_n(\xi)}(g)}{|Y_n(\xi)|}\frac{1_{Z_n(g)}(\xi')}{\nu(Z_n(g))}.$$
We will show that $\{\Upsilon_n\}$ is admissible (if $\rho$ is sufficiently large).
\end{defn}
 %To do this, we need the following lemmas, most of which are proven in the next sections. %To make notation easier we will write 
%$$x \doteq y \quad (C)$$
%if $|x - y| \le C$.

\begin{convention}
Given real numbers $x$ and $y$ we write $x \lessdot y$ if there is a constant $C>0$, depending only on $(\Gamma,d)$ and $\nu$, such that $x \le y + C$. We write $x \doteq y$ if both $x \lessdot y$ and $y \lessdot x$. We write $x \lesssim y$ if there is a constant $C>1$ depending only of $(\Gamma,d)$ and $\nu$, such that $x \le Cy$. We write $x \asymp y$ if both $x\lesssim y$ and $y \lesssim x$. 

We also write $x \lessdot_\rho y$ to indicate that the implied constant may depend on $\rho$. Thus $x \lessdot_\rho y$ means there is a constant $C>0$, depending only on $(\Gamma,d)$, $\nu$ and $\rho$ such that $x \le y+ C$. The notations $x \doteq_\rho y, x \lesssim_\rho y, x\asymp_\rho y$ are defined similarly.
\end{convention}

The three lemmas below are proven in the next three sections.
\begin{lem}\label{lem:vol2}
There exist constants $a_0, T_0>0$ such that for every $\xi \in \partial \Gamma$ if $a \ge a_0$, $r\ge \max(|T_1|, |T_2|)-2c$ and $T_2-T_1 \ge T_0$ then
$$ \left|S(e;r-a,r) \cap h_\xi^{-1}([T_1,T_2]) \right| \asymp e^{\fh (r+T_2)/2} $$
where $S(e;r-a,r) = \{g \in \Gamma:~ r-a < d(e,g)\le r\}$.
\end{lem}

\begin{lem}\label{lem:Z}
If $\Upsilon_n(g,\xi,\xi')>0$ and $\rho$ is sufficiently large (depending on $(\Gamma,d)$ and $\nu$) then
$$h_{g^{-1}\xi'}(g^{-1}) \gtrdot_\rho 0, \quad h_{g^{-1}\xi}(g^{-1}) \gtrdot_\rho 0, \quad h_{\xi'}(g) \in (-5\rho, \rho), \quad (\xi|\xi')_e \ge n - 5\rho, \quad (g^{-1}\xi|g^{-1}\xi')_e \ge n-10\rho.$$
\end{lem}

\begin{lem}\label{lem:Y}
If $\rho$ is sufficiently large (dependent only on $(\Gamma,d)$ and $\nu$) then for any $g\in \Gamma$ with $d(e,g) \in (2n-2\rho,2n)$, $\nu(Z_n(g)) \asymp_\rho e^{-\fh n}$.
% $$ |\{f \in Y_n(\xi):~f^{-1}\xi' \in Z_n(f)\}| \lesssim_\rho e^{\fh n}, \quad  |\{f \in \Gamma:~f \in Y_n(f^{-1}\xi)\}| \lesssim_\rho e^{\fh n}.$$
\end{lem}

\begin{proof}[Proof of Lemma \ref{lem:admissible} given Lemmas \ref{lem:vol2}-\ref{lem:Y}]
%Define $\Upsilon'_n:\Gamma \times \cH \times \cH \to \RR$ by
%$$\Upsilon'_n(g,h,h') := \frac{1_{Y_n(h)}(g)}{|Y_n(h)|}\frac{1_{Z_n(g,h)}(h')}{\nu(Z_n(g,h))}.$$
We assume $\rho>0$ is large enough so that the conclusions to the lemmas above hold. We check the conditions of Definition \ref{defn:zeta} one at a time. The first requirement is immediate:
$$\sum_{g\in \Gamma} \int \Upsilon_n(g,\xi,\xi') ~d\nu(\xi') = \sum_{g\in \Gamma} \int\frac{1_{Y_n(\xi)}(g)}{|Y_n(\xi)|}\frac{1_{Z_n(g)}(\xi')}{\nu(Z_n(g))} ~d\nu(\xi')=1.$$
 By Lemma \ref{lem:Z} if $(g,\xi,\xi')$ is such that $\Upsilon_n(g,\xi,\xi')>0$ then $(\xi|\xi')_e \ge n-5\rho$ and $(g^{-1}\xi|g^{-1}\xi')_e \ge n-10\rho$. This proves the second requirement (with respect to any visual metric). The  definition of $Y_n$ and Lemma \ref{lem:Z} imply that if $(g,\xi,\xi') \in S_n$ (the support of $\Upsilon_n$) then $|h_\xi(g)| + |h_{\xi'}(g)| \le 10\rho$.  By the definition of quasi-conformality, $R(g^{-1},\xi) \doteq -\fh h_\xi(g)$. So this proves  the third requirement.

To prove the fourth requirement, fix $n\ge 0$ and $\xi' \in \partial \Gamma$. By Lemma \ref{lem:Y},
\begin{eqnarray*}
 \int \sum_{g\in \Gamma}  \Upsilon_{n}(g,\xi,\xi')~d\nu(\xi) &=& \int  \frac{1}{|Y_n(\xi)|} \sum_{g\in Y_n(\xi)} \frac{1_{Z_n(g)}(\xi')}{\nu(Z_n(g))} ~d\nu(\xi)\\
 &\lesssim_\rho & e^{\fh n}  \int  \frac{1}{|Y_n(\xi)|} \sum_{g\in Y_n(\xi)} 1_{Z_n(g)}(\xi')  ~d\nu(\xi)\\
 &=&  e^{\fh n}  \int  \frac{|\{g\in Y_n(\xi):~\xi' \in Z_n(g)\}|}{|Y_n(\xi)|} ~d\nu(\xi)\\
 &\le&  e^{\fh n}  \nu(\{\xi:~ \exists g\in Y_n(\xi), \xi' \in Z_n(g)\})\\
  &\le& e^{\fh n}  \nu(\{\xi:~ (\xi|\xi')_e \ge n-5\rho\}) \lesssim_\rho 1. 
  \end{eqnarray*} 
The last two inequalities above follow from Lemmas \ref{lem:Z} and \ref{lem:conformal}. This proves the first inequality of the fourth requirement of Definition \ref{defn:zeta}.  

To prove the second inequality, note that Lemma \ref{lem:Z} implies that if $\Upsilon_{n}(g,\xi,g\xi')>0$ then $h_{\xi'}(g^{-1}) \gtrdot_\rho 0$. So the quasi-conformality of $\nu$ implies
 $$\frac{d\nu \circ g}{d\nu}(\xi') \asymp e^{-\fh h_{\xi'}(g^{-1})} \lesssim_\rho 1.$$
Lemma \ref{lem:vol2} implies $|Y_n(\xi)| \asymp_\rho e^{\fh n}$ and Lemma \ref{lem:Y} implies $\nu(Z_n(g)) \asymp_\rho e^{-\fh n}$. By Lemma \ref{lem:Z}, 
  \begin{eqnarray*}
 \int \sum_{g\in \Gamma}  \Upsilon_{n}(g,\xi,g\xi')\frac{d\nu \circ g}{d\nu}(\xi')~d\nu(\xi)   &\lesssim_\rho& \int \sum_{g\in \Gamma}  \Upsilon_{n}(g,\xi,g\xi')~d\nu(\xi)\\
  &=& \int  \frac{1}{|Y_n(\xi)|} \sum_{g\in Y_n(\xi)} \frac{1_{Z_n(g)}(g\xi')}{\nu(Z_n(g))} ~d\nu(\xi)\\
 &\lesssim_\rho &  \int   \sum_{g\in Y_n(\xi)} 1_{Z_n(g)}(g\xi')  ~d\nu(\xi)\\
 &=&  \sum_{g\in \Gamma} \nu(\{\xi |~ g\in Y_n(\xi), ~ g\xi' \in Z_n(g)\})\lesssim_\rho 1.
 \end{eqnarray*} 
To see the last inequality above, note that if $g \in Y_n(\xi)$ and $g\xi' \in Z_n(g)$ then Lemma \ref{lem:Z} implies $(\xi|g\xi')_e \gtrdot_\rho n$. So Lemma \ref{lem:conformal} implies $\nu(\{\xi |~ g\in Y_n(\xi), ~ g\xi' \in Z_n(g)\}) \lesssim_\rho e^{-\fh n}$. By Lemma \ref{lem:Z}, the number of nonzero terms in the sum is at most 
$$|\{g\in \Gamma:~ d(g,e) \in (2n-2\rho,2n), h_{g\xi'}(g) \in (-5\rho,\rho) \}|.$$
By (\ref{eqn:horo}) if $\{x_n\}$ is a sequence in $\Gamma$ converging to $\xi'$ then
\begin{eqnarray}\label{equation}
h_{g\xi'}(g) \doteq \liminf_{n\to\infty} d(gx_n,g)-d(gx_n,e) = \liminf_{n\to\infty} d(x_n,e)-d(x_n,g^{-1}) \doteq - h_{\xi'}(g^{-1}).
\end{eqnarray}
Therefore, the number of nonzero terms is at most
 $$|\{g\in \Gamma:~ d(g,e) \in (2n-2\rho,2n), h_{\xi'}(g^{-1}) \in (-6\rho,6\rho) \}|$$
 when $\rho$ is sufficiently large.  Lemma \ref{lem:vol2} implies that this is $\lesssim_\rho e^{\fh n}$. This proves the second inequality (of the 4th requirement of definition \ref{defn:zeta}). 

For the last inequality note, we use the quasi-conformality of $\nu$ to say that if $\Upsilon_{n}(g,g\xi,\xi')>0$ then Lemma \ref{lem:Z} implies
 $$\frac{d\nu \circ g}{d\nu}(\xi) \asymp e^{-\fh h_{\xi}(g^{-1})} \lesssim_\rho 1.$$
By  Lemma \ref{lem:vol2} $|Y_n(g\xi)| \lesssim_\rho e^{\fh n}$. So,
   \begin{eqnarray*}
 \int \sum_{g\in \Gamma}  \Upsilon_{n}(g,g\xi,\xi')\frac{d\nu \circ g}{d\nu}(\xi)~d\nu(\xi') &\lesssim_\rho& \int \sum_{g\in \Gamma}  \Upsilon_{n}(g,g\xi,\xi')~d\nu(\xi')\\
&=& \int \sum_{g\in \Gamma} \frac{1_{Y_n(g\xi)}(g) 1_{Z_n(g)}(\xi')}{|Y_n(g\xi)|\nu(Z_n(g))} ~d\nu(\xi')\\
&=&  \sum_{g\in \Gamma} \frac{1_{Y_n(g\xi)}(g)}{|Y_n(g\xi)|} \\
&\lesssim_\rho& e^{-\fh n} |\{g\in \Gamma:~ g\in Y_n(g\xi)\}|  \lesssim_\rho 1.
   \end{eqnarray*} 
 To see the last inequality above, note that if $g \in Y_n(g\xi)$ then $d(e,g) \in (2n-2\rho,2n)$ and $h_{g\xi}(g) \in (2\rho,4\rho)$. Equation (\ref{equation}) shows $h_{g\xi}(g) \doteq -h_\xi(g^{-1})$. So if $\rho$ is sufficiently large then Lemma \ref{lem:vol2} implies
 $$|\{g\in \Gamma:~ g\in Y_n(g\xi)\}| \le |\{g\in \Gamma:~d(e,g) \in (2n-2\rho,2n), ~h_\xi(g^{-1}) \in (-6\rho,6\rho)\}| \lesssim_\rho e^{\fh n}.$$
 This proves the last inequality. So $\{\Upsilon_n\}_{n=1}^\infty$ is an admissible family. 

Finally observe that by Lemma \ref{lem:Z}, if $(g,\xi,\xi') \in S_n$ then $|h_\xi(g)-h_{\xi'}(g)| \ge \rho$. By definition of quasi-conformality, $R(g^{-1},\xi) \doteq -\fh h_\xi(g)$. So by choosing $\rho$ sufficiently large, we can require that $|R(g^{-1},\xi) - R(g^{-1},\xi')| \ge Q$ for all $(g,\xi,\xi') \in S_n$ as required.

%So we choose $\rho$ so that $2\rho -4\delta \ge C_0$. This proves the last requirement of the lemma.
\end{proof}

\section{Proof of Lemma \ref{lem:vol2} }\label{sec:volume} %%%%%%%%%%%%%%  vol2  %%%%%%%%%%%%%%%%%%

%For this section, we assume $(X,d_X)$ is a proper $(1,c)$-quasi-geodesic $\delta$-hyperbolic metric space. Let $B(x,r)$ denote the closed ball of radius $r$ centered at $x$.

Throughout this section, we assume $(\Gamma,d)$ is a proper $(1,c)$-quasi-geodesic $\delta$-hyperbolic group. For $g \in \Gamma$ and $r>0$, let $B(g,r)$ denote the closed ball of radius $r$ centered at $g$. 

\begin{lem}
Let $p,q \in \Gamma$ and $r_p,r_q>0$ be such that $|r_p-r_q| \le d(p,q)-2c$. Then there exists an element $x \in \Gamma$ on a $(1,c)$-quasi-geodesic between $p$ and $q$ such that 
$$B\Big(x, \frac{r_p+r_q -d(p,q)}{2}-2c\Big) \subset B(p,r_p) \cap B(q,r_q) \subset B\Big(x, \frac{r_p+r_q -d(p,q)}{2} + 6c+2\delta\Big).$$
\end{lem}

\begin{proof}
Let $\gamma:[a,b] \to \Gamma$ be a $(1,c)$-quasi-geodesic from $p$ to $q$ for some interval $[a,b] \subset \RR$. Let $t=a + \frac{d(p,q)+r_p-r_q}{2}$. Since $\gamma$ is a $(1,c)$-quasi-geodesic, $|(b-a)-d(p,q)| \le c$. So $t \in [a,b]$. Let $x = \gamma(t)$. Note
$$\Big|d(x,p) - \frac{d(p,q)+r_p-r_q}{2}\Big| \le 2c \textrm{ and } \Big|d(x,q) - \frac{d(p,q)+r_q-r_p}{2}\Big| \le 2c.$$
The first inclusion now follows from the triangle inequality. 

 Let $y \in B(p,r_p) \cap B(q,r_q)$. By definition of $\delta$-hyperbolicity,
 $$2c \ge (p|q)_x \ge \min\{ (p|y)_x, (q|y)_x\} - \delta.$$
 Suppose that $\min\{ (p|y)_x, (q|y)_x\}  = (p|y)_x \le 2c+\delta$. Then
 \begin{eqnarray*}
 d(y,x) &=& 2(p|y)_x - d(p,x) +d(p,y) \le 4c+2\delta - \frac{d(p,q)+r_p-r_q}{2} + 2c + r_p\\
  &=& 6c+2\delta + \frac{r_p+r_q -d(p,q)}{2}.
  \end{eqnarray*}
By symmetry the above inequality holds if $\min\{ (p|y)_x, (q|y)_x\}  = (q|y)_x$. This proves the second inequality.
 \end{proof}

\begin{cor}\label{cor:ellipse}
 Let $\xi\in \partial \Gamma$. Let $0<r$ and $T \le r-2c$. Then there exists an element $x \in \Gamma$ such that
\begin{eqnarray*}
B\Big(x, \frac{r+T}{2}-2c\Big) &\subset& B(e,r) \cap h^{-1}_\xi(-\infty,T+4\delta]\\
 B(p,r) \cap h^{-1}_\xi(-\infty,T-4\delta]  &\subset& B\Big(x, \frac{r+T}{2} + 6c+2\delta\Big).
 \end{eqnarray*}
\end{cor}

\begin{proof}
Let $\{q_n\}_{n=0}^\infty$ be a sequence in $\Gamma$ converging to $\xi$. Let $r_n=d(e,q_n)+T$. Since $\lim_{n\to\infty} d(e,q_n) = +\infty$,  $|r - r_n| \le d(e,q_n)-2c$ for all sufficiently large $n$. So the previous lemma implies the existence of $x_n\in \Gamma$ such that 
$$B\Big(x_n, \frac{r+T}{2}-2c\Big) \subset B(e,r) \cap B(q_n,r_n) \subset B\Big(x_n, \frac{r+T}{2} + 6c+2\delta\Big).$$
Since $x_n \in B(e,r)$ for all $n$ and $\Gamma$ is proper, there exists a limit point $x$ of $\{x_n\}_{n=1}^\infty$ that lies in $B(e,r)$. The corollary now follows by taking limits as $n\to\infty$ and using (\ref{eqn:horo}). 
%Since $B(q_n,r_n)$ limits on $h^{-1}(-\infty,t]$ and 
\end{proof}

\begin{lem}\label{lem:vol1}
There exists a constants $T_0>0$ such that if $T_2\ge T_1$ are such that $T_2-T_1\ge T_0$, $\xi\in \partial \Gamma$ and $r\ge \max(|T_1|, |T_2|)-2c$ then
$$ |B(e,r) \cap h^{-1}_\xi[T_1,T_2] |  \asymp e^{\fh (r+T_2)/2}.$$
where $\fh >0$ is the quasi-conformality constant from Lemma \ref{lem:conformal}.
\end{lem}

\begin{proof}
%Let $h \in \cH(\Gamma,d)$ and $r,T>0$. 
By Corollary \ref{cor:ellipse} and Lemma \ref{lem:conformal}, for any $T$, if $r\ge |T|-2c$ then
$$|B(e,r) \cap h^{-1}_\xi(-\infty,T] | \asymp  e^{\fh (r+T/2)}.$$
So if $r\ge \max(|T_1|, |T_2|)-2c$ then
\begin{eqnarray*}
|B(e,r) \cap h^{-1}_\xi[T_1,T_2] | &\ge & |B(e,r) \cap h^{-1}_\xi(-\infty,T_2] | - |B(e,r) \cap h^{-1}_\xi(-\infty,T_1] |\\
&\gtrsim& e^{\fh (r+T_2/2)}
\end{eqnarray*}
if $T_2-T_1$ is sufficiently large. On the other hand, 
\begin{eqnarray*}
|B(e,r) \cap h^{-1}_\xi[T_1,T_2] | &\le & |B(e,r) \cap h^{-1}_\xi(-\infty,T_2] | - |B(e,r) \cap h^{-1}_\xi(-\infty,T_1-1] |\\
&\lesssim& e^{\fh (r+T_2/2)}
\end{eqnarray*}
if $T_2-T_1$ is sufficiently large.
\end{proof}

{\bf Lemma \ref{lem:vol2}.} {\em
There exist constants $a_0, T_0>0$ such that for every $\xi \in \partial \Gamma$ if $a \ge a_0$, $r\ge \max(|T_1|, |T_2|)-2c$ and $T_2-T_1 \ge T_0$ then
$$ \left|S(e;r-a,r) \cap h_\xi^{-1}[T_1,T_2] \right| \asymp e^{\fh (r+T_2)/2} $$
where $S(e;r-a,r) = \{g \in \Gamma:~ r-a < d(e,g)\le r\}$.}

\begin{proof}
Let $T_0>0$ be as in the previous lemma. The upper bound follows from the previous lemma and the inclusion $S(e;r-a,r) \subset B(e,r)$. Also
\begin{eqnarray*}
 |S(e;r-a,r) \cap h^{-1}_\xi[T_1,T_2] | &= & |B(e,r) \cap h^{-1}_\xi[T_1,T_2] |  -  |B(e,r-a) \cap h^{-1}_\xi[T_1,T_2] |  \\
 &\gtrsim& e^{\fh (r+T_2)/2} - e^{\fh (r+T_2-a)/2} \asymp e^{\fh (r+T_2)/2}
 \end{eqnarray*}
 if $a$ is sufficiently large. 
\end{proof}
%Lemma \ref{lem:vol2} follows immediately from the previous lemma.

%\section{Proof of Lemma \ref{lem:c3} }
\section{Proof of Lemma \ref{lem:Z} }

\begin{lem}\label{lem:only2}
For any $\xi,\xi' \in\partial \Gamma$ and any $g,f \in \Gamma,$
$$2(\xi|g)_f \doteq h_{\xi}(f)+d(g,f) - h_\xi(g).$$
$$2(\xi|\xi')_e + h_\xi(g) + h_{\xi'}(g) \doteq 2(\xi|\xi')_g \doteq 2(g^{-1}\xi|g^{-1}\xi')_e.$$
%So
%$$(\xi |g)_e \doteq d(e,g)/2 - R(g,\xi)/2 \quad (\delta).$$
%Thus if $\xi,\xi' \in \partial \Gamma$ then
%$$2(\xi|g)_f \doteq R(f,\xi)+d(g,f) - R(g,\xi) \quad (\delta),$$
%$$2(\xi|\xi')_e + R(g,\xi) + R(g,\xi')\doteq 2(\xi|\xi')_g \doteq 2(g^{-1}\xi|g^{-1}\xi')_e \quad (\delta).$$
\end{lem}

\begin{proof}
Let $x_i$ be a sequence in $\Gamma$ converging to $\xi$. By (\ref{eqn:gromovproduct}) and (\ref{eqn:horo}),
\begin{eqnarray*}
2(\xi|g)_f &\doteq&  \liminf_{n\to\infty} 2(x_n|g)_f =  \liminf_{n\to\infty} d(x_n,f)+d(g,f)-d(x_n,g)\\
&=&  \liminf_{n\to\infty} d(x_n,f)-d(x_n,e)+d(g,f)+d(x_n,e)-d(x_n,g) \doteq  h_{\xi}(f)-h_\xi(g)+d(g,f).
\end{eqnarray*}
Similarly, if $x'_i$ is a sequence in $\Gamma$ converging to $\xi'$ then
\begin{eqnarray*}
2(\xi|\xi')_g &\doteq& \liminf_{n,m\to\infty} 2(x_n|x_m')_g = \liminf_{n,m\to\infty} d(x_n,g)+d(x'_m,g)-d(x_n,x'_m)\\
&=&  \liminf_{n,m\to\infty} d(g^{-1}x_n,e)+d(g^{-1}x'_m,e)-d(g^{-1}x_n,g^{-1}x'_m) \doteq 2(g^{-1}\xi|g^{-1}\xi')_e.
\end{eqnarray*}
Similarly,
 \begin{eqnarray*}
2(\xi|\xi')_g &\doteq& \liminf_{n,m\to\infty} 2(x_n|x_m')_g =  \liminf_{n,m\to\infty} d(x_n,g)+d(x'_m,g)-d(x_n,x'_m)\\
&=&  \liminf_{n,m\to\infty} d(x_n,g) - d(x_n,e)+d(x'_m,g)-d(x'_m,e) +d(x_n,e) + d(x'_m,e) - d(x_n,x'_m)\\
&\doteq& h_\xi(g) + h_{\xi'}(g) + 2(\xi|\xi')_e.
\end{eqnarray*}
\end{proof}

%\section{Proof of Lemmas}

{\bf Lemma \ref{lem:Z}}. {\em
If $\Upsilon_n(g,\xi,\xi')>0$ and $\rho$ is sufficiently large (depending on $(\Gamma,d)$ and $\nu$) then
$$h_{g^{-1}\xi'}(g^{-1}) \gtrdot_\rho 0, \quad h_{g^{-1}\xi}(g^{-1}) \gtrdot_\rho 0, \quad h_{\xi'}(g) \in (-5\rho, \rho), \quad (\xi|\xi')_e \ge n - 5\rho, \quad (g^{-1}\xi|g^{-1}\xi')_e \ge n-10\rho.$$}
\begin{proof}
Because $\Upsilon_n(g,\xi,\xi')>0$, we have
$$d(e,g) \in (2n-2\rho, 2n),  \quad h_\xi(g) \in (2\rho,4\rho), \quad (\xi'|g)_e \in (n,n+\rho).$$
By Lemma \ref{lem:only2}
$$2(\xi'|g)_e \doteq d(e,g) - h_{\xi'}(g).$$
So 
$$-4\rho \lessdot h_{\xi'}(g)  \lessdot 0.$$
So $h_{\xi'}(g) \in (-5\rho, \rho)$ (assuming $\rho$ is sufficiently large). 

Let $\{x_n\}$ be a sequence in $\Gamma$ converging to $\xi'$. By (\ref{eqn:horo}), 
$$h_{g^{-1}\xi'}(g^{-1})\doteq \liminf_{n\to\infty} d(g^{-1}x_n,g^{-1}) - d(g^{-1}x_n,e) =  \liminf_{n\to\infty} d(x_n,e) - d(x_n,g) \doteq -h_{\xi'}(g) \ge -\rho.$$
This shows $h_{g^{-1}\xi'}(g^{-1}) \gtrdot_\rho 0$. The proof that $h_{g^{-1}\xi}(g^{-1}) \gtrdot_\rho 0$ is similar.

By Lemma \ref{lem:only2} $2(\xi|g)_e \doteq d(e,g) - h_\xi(g)$. So
$$n - 3 \rho \lessdot (\xi|g)_e \lessdot n-\rho.$$
By (\ref{eqn:gromov}), 
$$(\xi|\xi')_e \gtrdot \min\{ (\xi|g)_e, (\xi'|g)_e\} \gtrdot n-3\rho.$$
This proves $(\xi|\xi')_e \ge n - 5\rho$ (for sufficiently large $\rho$). To prove the last inequality, observe that by Lemma \ref{lem:only2},
$$2(g^{-1}\xi|g^{-1}\xi')_e \doteq 2(\xi|\xi')_e + h_\xi(g) + h_{\xi'}(g).$$
So the previous inequalities imply $(g^{-1}\xi|g^{-1}\xi')_e \ge n-10\rho$ if $\rho$ is sufficiently large.
\end{proof}

\section{Proof of Lemma \ref{lem:Y}}

\begin{lem}\label{lem:exist0}
For any $x \in \Gamma \cup \partial \Gamma$ there exists a $\xi \in \partial \Gamma$ with $(x|\xi)_e \doteq 0$.
\end{lem}

\begin{proof}
If the lemma is false then for every $n$ there exists $x_n \in \Gamma \cup \partial \Gamma$ such that $(\xi|x_n)_e \ge n$ for all $\xi \in \partial \Gamma$. By (\ref{eqn:gromov}), for any $\xi,\xi' \in \partial \Gamma$, 
$$(\xi | \xi')_e \gtrdot \min \{ (\xi|x_n)_e, (\xi'|x_n)_e \}\ge n$$
which implies $(\xi | \xi')_e = +\infty$ for all $\xi, \xi' \in \partial \Gamma$. But $(\xi | \xi')_e = +\infty$ implies $\xi = \xi'$ so $\partial \Gamma$ consists of only a single point. This contradicts our assumption that $\Gamma$ is non-elementary and therefore the boundary is infinite.
\end{proof}

\begin{lem}\label{lem:exist1}
There exists a constant $C>0$ such that if $g \in \Gamma, r>0$ and $r +C < d(e,g) $ then there exists $\xi''\in \partial \Gamma$ with $(\xi'' | g)_e \doteq r$.
\end{lem}

\begin{proof}
By the previous lemma there exists an element $\xi' \in \partial \Gamma$ with $(\xi' | g^{-1})_e \doteq 0$. By Lemma \ref{lem:only2},
$$0 \doteq 2(\xi' | g^{-1})_e \doteq d(e,g) - h_{\xi'}(g^{-1}).$$
By another application of Lemma \ref{lem:only2} and (\ref{equation}) this implies
\begin{eqnarray}\label{eqn:g}
2(g\xi' | g)_e \doteq d(e,g) - h_{g\xi'}(g) \doteq d(e,g) + h_{\xi'}(g^{-1}) \doteq 2d(e,g).
\end{eqnarray}

Let $\gamma:[0,\infty] \to \overline{\Gamma}$ be a $(1,c)$-quasi-geodesic with $\gamma(0)=e$ and $\gamma(\infty) =g\xi'$. Let $f= \gamma(r)$. So $d(e,f)\doteq r$. By (\ref{eqn:horo})
\begin{eqnarray*}
h_{g\xi'}(f) &\doteq & \liminf_{t\to\infty} d(\gamma(t), f) - d(\gamma(t),e) =\liminf_{t\to\infty} d(\gamma(t), \gamma(r)) - d(\gamma(t),e)\\
&\doteq& \liminf_{t\to\infty} t-r -t = -r.
\end{eqnarray*}
By Lemma \ref{lem:only2},
$$(f|g\xi')_e \doteq (1/2)( d(e,f) - h_{g\xi'}(f) ) \doteq r.$$
By Lemma \ref{lem:exist0} there exists an element $\xi \in \partial \Gamma$ with $(\xi | f^{-1}g\xi')_e \doteq 0$. We will show that $\xi'' := f\xi$ satisfies the lemma.

Since $|h_{f\xi}(f)| \le d(e,f)$,
\begin{eqnarray*}
(f\xi|f)_e \doteq (1/2)( d(e,f) - h_{f\xi}(f) ) \le d(e,f) \doteq r.
\end{eqnarray*}
By (\ref{eqn:gromov}), 
$$(f\xi | g\xi')_e \gtrdot \min \{ (f\xi | f)_e, (f | g\xi')_e \} \gtrdot r.$$
An application of Lemma \ref{lem:only2} and $(\xi | f^{-1}g\xi')_e \doteq 0$ yields
\begin{eqnarray}\label{eqn:thing2}
2r \lessdot 2(f \xi | g\xi')_e \doteq 2(\xi | f^{-1}g\xi')_e - h_{f\xi}(f) - h_{g\xi'}(f) \doteq  - h_{f\xi}(f) +r.
\end{eqnarray}
So $h_{f\xi}(f) \lessdot -r$. Since $h_{f\xi}(f) \ge - d(e,f) \doteq -r$ we have $h_{f\xi}(f) \doteq -r$. So (\ref{eqn:thing2}) implies
$$(f\xi | g\xi')_e \doteq r.$$
Since $(g\xi' | g)_e \doteq d(e,g)$ (by (\ref{eqn:g})), there is a constant $C>0$ so that if $d(e,g) > r+C$ then $(f\xi | g\xi')_e \le (g\xi' | g)_e$. Assuming this, we have by (\ref{eqn:gromov}),
$$(f\xi | g)_e \gtrdot \min \{ (f\xi| g\xi')_e, (g\xi' | g)_e \} = (f\xi | g\xi')_e \doteq r.$$
Also, 
$$r \doteq (f\xi| g\xi')_e\gtrdot \min \{(f\xi | g)_e  , (g\xi' | g)_e \}.$$
Since $r +C\le d(e,g) \doteq (g\xi' | g)_e$, we must have that if $C>0$ is large enough then $r \gtrdot (f\xi | g)_e$. As we have already shown the opposite inequality, we now have $r \doteq (f\xi | g)_e$. This proves the lemma by setting $\xi'':=f\xi$. 
\end{proof}

{\bf Lemma \ref{lem:Y}}. {\em
If $\rho$ is sufficiently large (dependent only on $(\Gamma,d)$ and $\nu$) then for any $g\in \Gamma$ with $d(e,g) \in (2n-2\rho,2n)$, $\nu(Z_n(g)) \asymp_\rho e^{-\fh n}$. }

\begin{proof}

%The first inequality is an immediate consequence of Lemma \ref{lem:vol2}.
Let $\xi,\xi' \in Z_n(g)$ (i.e., $(\xi|g)_e, (\xi'|g)_e \in (n,n+\rho)$). By (\ref{eqn:gromov}),
$$(\xi|\xi')_e \gtrdot \min\{ (\xi|g)_e, (\xi'|g)_e\} >n.$$
So, fixing $\xi$, we see that $Z_n(g) \subset \{\xi':~ (\xi|\xi')_e \ge n-\rho\}$ if $\rho$ is sufficiently large. By Lemma \ref{lem:conformal}, the later set has measure $\lesssim_\rho e^{-\fh n}$. This shows $\nu(Z_n(g)) \lesssim_\rho e^{-\fh n}$. 

To prove the opposite inequality, by Lemma \ref{lem:exist1}, there exists an element $\xi'\in \partial \Gamma$ satisfying $n+\rho/3 \le (\xi'|g)_e \le n+\rho/2$ (if $\rho$ is sufficiently large). Let $\xi'' \in \partial \Gamma$ be such that $(\xi' | \xi'') \ge n+\rho$. By (\ref{eqn:gromov}),
$$(\xi'' | g)_e \gtrdot \min\{ (\xi' | \xi'')_e, (\xi' | g)_e \} \ge n + \rho/3.$$
Also,
$$(\xi'|g)_e \gtrdot \min\{ (\xi'' | g)_e, (\xi'' |\xi')_e \}.$$
If $\rho$ is sufficiently large, then it cannot be true that $n+\rho/2 \ge (\xi'|g)_e \gtrdot (\xi''|\xi')_e \ge n+\rho$. So we must have that 
$$n+\rho/2 \ge (\xi'|g)_e \gtrdot (\xi'' |g)_e.$$
We have now shown that
$$n+\rho/2 \gtrdot (\xi''|g)_e \gtrdot n+\rho/3.$$
It follows that if $\rho$ is large enough then $(\xi''|g)_e \in (n,n+\rho)$, i.e., $\xi'' \in Z_n(g)$. Thus
$$Z_n(g) \supset \{\xi''\in\partial \Gamma:~ (\xi''|\xi')_e \ge n+\rho\}.$$
It follows from Lemma \ref{lem:conformal} that $\nu(Z_n(g)) \gtrsim_\rho e^{-\fh n}$. This proves the opposite inequality.
 
\end{proof}

\end{document}